\documentclass[runningheads]{llncs}
\usepackage{latexsym,relsize}
\usepackage{stmaryrd}
\usepackage{graphicx}
\usepackage{amsmath,amsfonts,amssymb}
\usepackage{mathtools}
\usepackage{longtable}
\usepackage{algorithm}
\usepackage{algpseudocode}
\usepackage{comment}
\usepackage{color}
\usepackage{bussproofs}
\EnableBpAbbreviations

\usepackage{tikz}
\usepackage{colonequals}

\usepackage{multirow}

\usepackage{multicol}

\usepackage{array}
\newcolumntype{C}[1]{>{\centering\arraybackslash}p{#1}}
\newcolumntype{L}[1]{>{\arraybackslash}p{#1}}

\begin{document}
\def\mc{\multicolumn}
\newcommand{\cupdot}{\mathbin{\mathaccent\cdot\cup}}
\newcommand\xrowht[2][0]{\addstackgap[.5\dimexpr#2\relax]{\vphantom{#1}}}
\newcommand{\redbf}[1]{\textcolor{red}{\textbf{#1}}}
\newcommand{\red}[1]{\textcolor{red}{#1}}
\newcommand{\bluebf}[1]{\textcolor{blue}{\textbf{#1}}}
\newcommand{\blue}[1]{\textcolor{blue}{#1}}
\newcommand{\marginred}[1]{\marginpar{\raggedright\tiny{\red{#1}}}}
\newcommand{\marginredbf}[1]{\marginpar{\raggedright\tiny{\redbf{#1}}}}
\newcommand{\redfootnote}[1]{\footnote{\redbf{#1}}}
\def\fCenter{{\mbox{$\ \vdash\ $}}}
\newcommand{\fns}{\footnotesize}
\newcommand{\sabine}[1]{\textcolor{red}{\textbf{#1}}}
\newcommand{\willem}[1]{\textcolor{purple}{\textbf{#1}}}

\newcommand{\Abox}{\textsc{Abox} }
\newcommand{\Tbox}{\textsc{Tbox} }
\newcommand{\Aboxes}{\textsc{Abox}es }
\newcommand{\Tboxes}{\textsc{Tbox}es }
\renewcommand{\P}{\mathcal{P}}
\newcommand{\C}{\mathsf{C}}
\renewcommand{\c}{\mathsf{c}}
\renewcommand{\d}{\mathsf{d}}
\newcommand{\up}[1]{#1^{\uparrow}}
\newcommand{\down}[1]{#1^{\downarrow}}
\newcommand{\ud}[1]{#1^{\uparrow\downarrow}}
\newcommand{\du}[1]{#1^{\downarrow\uparrow}}
\newcommand{\D}{\mathsf{D}}
\newcommand{\I}{\mathsf{I}}
\newcommand{\q}{\mathsf{q}}
\renewcommand{\u}{\mathsf{u}}
\renewcommand{\L}{\mathbb{L}}
\newcommand{\F}{\mathbb{F}}
\newcommand{\osigma}{\overline{\sigma}}
\renewcommand{\emph}{\textbf}
\newcommand{\ovR}[1]{\overline{#1}^R}
\newcommand{\ovS}[1]{\overline{#1}^S}

\newcommand{\ovX}{\overline{X}}
\newcommand{\ovY}{\overline{Y}}
\newcommand{\ovg}{\overline{g}}
\newcommand{\ovf}{\overline{f}}

\newcommand{\op}{\mathbf{op}}
\newcommand{\mb}{\mathbb}
\newcommand{\mac}{\mathcal}

\newcommand{\Prop}{\mathsf{Prop}}
\newcommand{\Nom}{\mathsf{Nom}}
\newcommand{\Cnom}{\mathsf{Cnom}}

\newcommand{\low}{\mathsf{l}}
\newcommand{\cl}{\mathsf{c}}

\newcommand{\jty}{J^{\infty}}
\newcommand{\mty}{M^{\infty}}

\newcommand{\NOMI}{\mathbf{I}}
\newcommand{\NOMJ}{\mathbf{J}}
\newcommand{\NOMH}{\mathbf{H}}
\newcommand{\NOMK}{\mathbf{K}}
\newcommand{\NOMA}{\mathbf{A}}
\newcommand{\nomi}{\mathbf{i}}
\newcommand{\nomj}{\mathbf{j}}
\newcommand{\nomh}{\mathbf{h}}
\newcommand{\nomk}{\mathbf{k}}
\newcommand{\noma}{\mathbf{a}}

\newcommand{\NCL}{\mathbf{L}}

\newcommand{\CNOMM}{\mathbf{M}}
\newcommand{\CNOMN}{\mathbf{N}}
\newcommand{\CNOMO}{\mathbf{O}}
\newcommand{\CNOMX}{\mathbf{X}}
\newcommand{\cnomm}{\mathbf{m}}
\newcommand{\cnomn}{\mathbf{n}}
\newcommand{\cnomx}{\mathbf{x}}

\newcommand{\lcc}{\mathbf{c}}
\newcommand{\Leq}{\mathbf{s}}
\newcommand{\nLeq}{\mathbf{n}}

\newcommand{\Cc}{\mathbb{C}}

\newcommand{\p}{\mathcal{P}}

\newcommand{\marginnote}[1]{\marginpar{\raggedright\tiny{#1}}} 
\newcommand{\mpcolor}{\color{red}}
\newcommand{\mpprefix}{MP: }
\newcommand{\mpnote}[1]{{\mpcolor \mpprefix #1 }}
\newcommand{\mpmnote}[1]{\marginnote{\mpnote{#1}}}
\newcommand{\mpfnote}[1]{\footnote{\mpnote{#1}}}

\renewcommand{\L}{\mathcal{L}}
\newcommand{\Lb}{\mathcal{L}_\Box}
\newcommand{\Lmu}{\mathcal{L}_C}
\newcommand{\Ag}{\mathsf{Ag}}
\newcommand{\Cat}{\mathsf{Cat}}
\newcommand{\Ca}{\mathsf{C}_a}
\newcommand{\Cka}{\mathsf{C}_{k,a}}

\newcommand{\nomb}{\mathbf{b}}
\newcommand{\nomx}{\mathbf{x}}
\newcommand{\val}[1]{[\![{#1}]\!]}
\newcommand{\descr}[1]{(\![{#1}]\!)}
\renewcommand{\phi}{\varphi}


\newcommand{\wbox}{\ensuremath{\Box}\xspace}
\newcommand{\wdia}{\ensuremath{\Diamond}\xspace}
\newcommand{\bbox}{\ensuremath{\blacksquare}\xspace}
\newcommand{\bdia}{\ensuremath{\Diamondblack}\xspace}
\newcommand{\bcirc}{\ensuremath{{\bullet}}\xspace}
\newcommand{\WBOX}{\ensuremath{\check{\Box}}\xspace}
\newcommand{\WDIA}{\ensuremath{\hat{\Diamond}}\xspace}
\newcommand{\BBOX}{\ensuremath{\check{\blacksquare}}\xspace}
\newcommand{\BDIA}{\ensuremath{\hat{\Diamondblack}}\xspace}
\newcommand{\WCIR}{\ensuremath{\tilde{{\circ}}}\xspace}
\newcommand{\BCIR}{\ensuremath{\tilde{{\bullet}}}\xspace}

\newcommand{\aneg}{\ensuremath{\neg}\xspace}
\newcommand{\aatop}{\ensuremath{\top}\xspace}
\newcommand{\abot}{\ensuremath{\bot}\xspace}
\newcommand{\aand}{\ensuremath{\wedge}\xspace}
\newcommand{\aor}{\ensuremath{\vee}\xspace}
\newcommand{\ararr}{\ensuremath{\rightarrow}\xspace}
\newcommand{\Ararr}{\ensuremath{\Rightarrow}\xspace}
\newcommand{\alarr}{\ensuremath{\leftarrow}\xspace}
\newcommand{\adrarr}{\ensuremath{\,{>\mkern-7mu\raisebox{-0.065ex}{\rule[0.5865ex]{1.38ex}{0.1ex}}}\,}\xspace}
\newcommand{\adlarr}{\ensuremath{\rotatebox[origin=c]{180}{$\,>\mkern-8mu\raisebox{-0.065ex}{\aol}\,$}}\xspace}
%
\newcommand{\ANEG}{\ensuremath{\:\tilde{\neg}}\xspace}
\newcommand{\ATOP}{\hat{\top}}
\newcommand{\AATOP}{\hat{\top}}
\newcommand{\ABOT}{\ensuremath{\check{\bot}}\xspace}
\newcommand{\AAND}{\ensuremath{\:\hat{\wedge}\:}\xspace}
\newcommand{\AOR}{\ensuremath{\:\check{\vee}\:}\xspace}
\newcommand{\ABIGAND}[2]{\ensuremath{\:\hat{\bigwedge_{#1}^{#2}}\:}\xspace}
\newcommand{\ABIGOR}[2]{\ensuremath{\:\check{\bigvee_{#1}^{#2}}\:}\xspace}
\newcommand{\ARARR}{\ensuremath{\:\check{\rightarrow}\:}\xspace}
\newcommand{\ALARR}{\ensuremath{\:\check{\leftarrow}\:}\xspace}
\newcommand{\ADRARR}{\ensuremath{\hat{{\:{>\mkern-7mu\raisebox{-0.065ex}{\rule[0.5865ex]{1.38ex}{0.1ex}}}\:}}}\xspace}
\newcommand{\ADLARR}{\ensuremath{\:\hat{\rotatebox[origin=c]{180}{$\,>\mkern-8mu\raisebox{-0.065ex}{\aol}\,$}}\:}\xspace}

\def\fCenter{\vdash}

\newcommand{\vd}{\ \,\textcolor{black}{\vdash}\ \,}

\newcommand{\nx}{\textcolor{black}{x}}
\newcommand{\NR}{\textcolor{black}{R}}
\newcommand{\ny}{\textcolor{black}{y}}
\newcommand{\ol}[1]{\overline{#1}}

\newcommand\mAND{\,\bigcap\,}
\def\mBAND{\mbox{\,\raisebox{0ex}{\rotatebox[origin=c]{180}{$\bigsqcup$}}\,}}
\def\mRA{\mbox{\,\raisebox{0ex}{\rotatebox[origin=c]{-90}{$\bigcap$}}\,}}
\def\mBRA{\mbox{\,\raisebox{0ex}{\rotatebox[origin=c]{90}{$\bigsqcup$}}\,}}

\newcommand{\rulespace}{3.4mm}

\newcommand{\Diamondblack}{\blacklozenge}

\newcommand{\fakeparagraph}[1]{

\textit{#1} \ \ }

\newcommand{\cceq}{\coloncolonequals}
\newcommand{\ceq}{\colonequals}

\title{Non-distributive description logic\thanks{This paper is partially  funded by the EU MSCA (grant No.~101007627). The first author is funded by the National Research Foundation of South Africa (grant No.~140841). The third and fourth authors are partially funded by the NWO grant KIVI.2019.001.
}}
%
%
\author{Ineke van der Berg\inst{1,2}\orcidID{0000-0003-2220-1383} \and
Andrea De Domenico\inst{1}\orcidID{0000-0002-8973-7011} \and
Giuseppe Greco\inst{1}\orcidID{0000-0002-4845-3821} \and 
Krishna B. Manoorkar\inst{1}\orcidID{0000-0003-3664-7757}\and
Alessandra Palmigiano \inst{1,3}\orcidID{0000-0001-9656-7527} \and
Mattia Panettiere \inst{1}\orcidID{0000-0002-9218-5449} 
}
\authorrunning{van der Berg, De Domenico, Greco, Manoorkar, Palmigiano, Panettiere.}
%
\institute{Vrije Universiteit Amsterdam, The Netherlands \and
Department of Mathematical Sciences, Stellenbosch University, South Africa \and
Department of Mathematics and Applied Mathematics, University of Johannesburg, South Africa
}%
\maketitle              
\begin{abstract}
We define LE-$\mathcal{ALC}$, a  generalization of the description logic $\mathcal{ALC}$  based on the propositional logic of general (i.e.~not necessarily distributive) lattices, and semantically interpreted on relational structures  based on formal contexts from Formal Concept Analysis (FCA).  The description logic LE-$\mathcal{ALC}$ allows us to formally describe databases with objects, features, and formal concepts, represented according to FCA as Galois-stable sets of objects and features. We describe ABoxes and TBoxes in LE-$\mathcal{ALC}$,  provide a tableaux algorithm for checking the consistency of LE-$\mathcal{ALC}$ knowledge bases with acyclic TBoxes, and show its termination, soundness and completeness. Interestingly,  consistency checking for LE-$\mathcal{ALC}$ is in \textsc{PTIME} for acyclic and completely unravelled TBoxes, while the analogous problem in the classical $\mathcal{ALC}$ setting is \textsc{PSPACE}-complete.
\end{abstract}
\setlength{\abovedisplayskip}{0.2mm}
\setlength{\belowdisplayskip}{0.2mm}
\setlength{\abovedisplayshortskip}{0.2mm}
\setlength{\belowdisplayshortskip}{0.2mm}

\section{Introduction}\label{Sec:Introduction}
Description Logic (DL) \cite{DLhandbook} is a class of logical formalisms, typically based on classical first-order logic, and widely used in  Knowledge Representation and Reasoning to describe and reason about relevant concepts in a given application domain and their relationships. Since certain laws of classical logic fail in certain application domains, in recent years, there has been a growing interest in developing versions of description logics on  weaker (non-classical)  propositional bases. For instance, in \cite{depaiva2011constructive}, an intuitionistic version of the DL $\mathcal{ALC}$ has been introduced  for resolving some inconsistencies arising from the classical law  of excluded middle when applying $\mathcal{ALC}$ to legal domains. In \cite{fuzzyDL1,fuzzyDL2}, many-valued (fuzzy) description logics have been introduced to account for uncertainty and imprecision in processing  information in the Semantic Web, and recently, frameworks of non-monotonic description logics have been introduced \cite{giordano2015semantic,non-monotone,giordano2022conditional}.

One domain of application in which there is no consensus as to  how classical logic should be applied  is Formal Concept Analysis (FCA). In this setting, formal concepts arise from formal contexts
 $\mathbb{P}=(A,X,I)$, where $A$ and $X$ are sets (of objects and features respectively), and $I\subseteq A \times X$. 
 Specifically, formal concepts are represented as Galois-stable tuples $(B,Y)$ such that $B \subseteq A$ and $Y \subseteq X$ and $B = \{a \in A \mid \forall y (y \in Y \Rightarrow a I y)\}$ and $Y= \{x \in X \mid \forall b(b \in B \Rightarrow b I x)\}$. The formal concepts arising from a formal context are naturally endowed with a partial order (the sub-concept/super-concept relation) as follows: $(B_1, Y_1) \leq (B_2, Y_2)$ iff $B_1 \subseteq B_2$ iff $Y_2 \subseteq Y_1$.
 This partial order is a complete lattice, which is in general non-distributive. The failure of distributivity in the lattice of formal concepts introduces a tension between classical logic and the natural logic of formal concepts in FCA. 
This failure motivated the introduction  of
  lattice-based propositional (modal) logics  as  the (epistemic) logics of formal concepts \cite{conradie2016categories,conradie2017toward}. 
Complete relational  semantics of these logics  is given by {\em enriched formal contexts} (cf.~Section \ref{sssec:relsem}), relational structures $\mathbb{F}= (\mathbb{P}, \mathcal{R}_\Box, \mathcal{R}_\Diamond)$ based on formal contexts.\\ 
 In this paper, we  introduce LE-$\mathcal{ALC}$, a  lattice-based version of $\mathcal{ALC}$  which stands in the same relation 
 to the lattice-based modal logic of formal concepts \cite{conradie2020non} as classical $\mathcal{ALC}$ stands in relation to classical modal logic: 
 the language and semantics of LE-$\mathcal{ALC}$  is based on enriched formal contexts and their associated modal algebras. Thus, just like the language of $\mathcal{ALC}$ can be seen as a hybrid modal logic language interpreted on Kripke frames, the language of LE-$\mathcal{ALC}$ can be regarded as a hybrid modal logic language interpreted on enriched formal contexts. 
 
 FCA and DL are  different and well known approaches in the formal representation of  concepts (or categories). They have been  used together for several purposes \cite{DLandFCA1,DLandFCA2,DLandFCA3}.  Thus, providing a DL framework which allows us to describe formal contexts (possibly enriched, e.g.~with additional relations on them) would be useful in relating these frameworks both at a theoretical and at a practical level.   
 Proposals to connect FCA and DL have been made, in which  concept lattices serve as  models for DL concepts. Shilov and Han \cite{shilov2007proposal}  interpret the positive fragment of $\mathcal{ALC}$ concepts over concept lattices and show that this interpretation is compatible with standard Kripke models for  $\mathcal{ALC}$. A similar approach is used by Wrum \cite{wurm2017language} in which complete semantics for the (full) Lambek calculus is defined on concept lattices. The approach of the present paper for defining and interpreting non-distributive description logic and modal logic in relation with concept lattices with operators differs from the approaches mentioned above in that it is based on duality-theoretic insights (cf.~\cite{conradie2016categories}). This allows us not only  to  show that the DL framework introduced in the present paper is consistent with the standard DL setting and its interpretation on Kripke models, but also to show that several properties of these  logics and the meaning of their formulas  can also be ``lifted" from the classical (distributive)  to  non-distributive settings (cf.~\cite{conradie2022modal,conradie2020non,conradie2021rough} for extended discussions). 

The main technical contribution of this paper is a tableaux algorithm for checking the consistency of LE-$\mathcal{ALC}$ ABoxes. We show that the  algorithm is terminating, sound and complete.  Interestingly,   this algorithm has a polynomial time complexity, compared to the complexity of the consistency  checking of classical $\mathcal{ALC}$ ABoxes which is \textsc{PSPACE}-complete. This means that the consistency checking problem for completely unravelled TBoxes (cf.~Subsection \ref{ssec:Description logic ALC}) for our logic is in \textsc{PTIME}. The algorithm also  constructs a model for the given ABox which is polynomial in size. Thus, it also implies that the corresponding hybrid modal logic has the finite model property. \\
{\em Structure of the paper.}
In Section \ref{Sec:Preliminaries}, we give the necessary preliminaries on the DL $\mathcal{ALC}$,  lattice-based modal logics and their relational semantics. In Section \ref{Sec: LE-DL}, we introduce the syntax and the semantics of LE-$\mathcal{ALC}$. In Section \ref{Sec: tableau}, we introduce a tableaux algorithm for checking the  consistency  of LE-$\mathcal{ALC}$ ABoxes and show that it is terminating, sound and complete. In Section \ref{Sec: conclusions}, we conclude and discuss some   future research directions. 
\section{Preliminaries}\label{Sec:Preliminaries}
\subsection{Description logic $\mathcal{ALC}$}\label{ssec:Description logic ALC}

Let $\mathcal{C}$ and $\mathcal{R}$ be disjoint sets of primitive or atomic {\em concept names} and {\em role names}. The set of {\em concept descriptions}   over  $\mathcal{C}$  and $\mathcal{R}$ are defined recursively as follows.

{{\centering\small
$C \ceq A \mid \top \mid \bot \mid C \wedge C \mid C \vee C \mid \neg C \mid \exists r.C\mid \forall r.C$
\par}}

 \noindent where $A \in \mathcal{C}$ and $r \in \mathcal{R}$.  An {\em interpretation} is a tuple 
 $\mathrm{I}=(\Delta^{\mathrm{I}}, \cdot^{\mathrm{I}})$ s.t.~$\Delta^{\mathrm{I}}$ is a non-empty set and $\cdot^{\mathrm{I}}$  maps  every concept  $A \in \mathcal{C}$ to a set $A^{\mathrm{I}} \subseteq \Delta^{\mathrm{I}}$, and every role name $r \in \mathcal{R}$ to a relation $r^{\mathrm{I}} \subseteq \Delta^{\mathrm{I}} \times \Delta^{\mathrm{I}}$.
This mapping extends to all concept descriptions as follows:
\smallskip

{{\centering\small
\begin{tabular}{rclcrcl}
     $\top^{\mathrm{I}}$& = &$\Delta^{\mathrm{I}}$ & \quad &
     $\bot^{\mathrm{I}}$& = &$\varnothing$ \\ 
     $(C \wedge D)^{\mathrm{I}}$& = & $C ^{\mathrm{I}} \cap D ^{\mathrm{I}}$ & \quad &
    $(C \vee D)^{\mathrm{I}}$& = & $C ^{\mathrm{I}} \cup D ^{\mathrm{I}}$\\
    $(\exists r.C )^{\mathrm{I}}$& = & $\{ d \in \Delta^{\mathrm{I}} \mid \exists e((d,e) \in r^{\mathrm{I}}\ \&\ e \in C ^{\mathrm{I}}\}$ & \quad & $(\neg C )^{\mathrm{I}}$& = & $\Delta^{\mathrm{I}} \setminus C ^{\mathrm{I}}$
    \\
     $(\forall r.C )^{\mathrm{I}}$& = & $\{ d \in \Delta^{\mathrm{I}} \mid \forall e((d,e) \in r^{\mathrm{I}} \Rightarrow  e \in C ^{\mathrm{I}}\}$\\    
\end{tabular}
\par}}
\smallskip

 Let $\mathcal{S}$ be a set of individual names disjoint from $\mathcal{C}$ and $\mathcal{R}$, such that for every $a$ in $\mathcal{S}$, $a^\mathrm{I} \in \Delta^\mathrm{I}$. For any $a,b \in \mathcal{S}$, any   $C\in \mathcal{C}$ and  $r \in \mathcal{R}$, an expression of the form 
$a:C$ (resp.~$(a,b):r$) is  an $\mathcal{ALC}$ {\em concept assertion} (resp.~{\em role assertion}). A finite set of $\mathcal{ALC}$ concept and  role assertions is  an {\em $\mathcal{ALC}$ ABox}. An assertion $a:C$  (resp.~$(a,b):r$) is {\em satisfied} in an interpretation ${\mathrm{I}}$ 
  if $a^{\mathrm{I}} \in C^{\mathrm{I}}$ (resp.~if $(a^{\mathrm{I}},b^{\mathrm{I}}) \in r^{\mathrm{I}}$). 
 An $\mathcal{ALC}$ {\em TBox} is a finite set of expressions of the form $C_1 \equiv C_2$. An interpretation ${\mathrm{I}}$ {\em satisfies}  $C_1 \equiv C_2$ iff $C_1^{\mathrm{I}} = C_2^{\mathrm{I}}$. An $\mathcal{ALC}$ {\em knowledge base} is a tuple $(\mathcal{A}, \mathcal{T})$, where $\mathcal{A}$ is an $\mathcal{ALC}$ ABox, and  $\mathcal{T}$ is an $\mathcal{ALC}$ TBox. 
An interpretation ${\mathrm{I}}$ is a {\em model} for a knowledge base $(\mathcal{A}, \mathcal{T})$ iff it satisfies all members of $\mathcal{A}$ and $\mathcal{T}$.  A knowledge base $(\mathcal{A}, \mathcal{T})$ is {\em consistent} if there is a model for it. An ABox $\mathcal{A}$ (resp.~TBox $\mathcal{T}$) is  {\em consistent} if the knowledge base  $(\mathcal{A}, \varnothing)$ (resp.~$(\varnothing, \mathcal{T})$) is consistent.

An $\mathcal{ALC}$ {\em concept deﬁnition} in $T$ is an expression of the form $A \equiv C$ where $A$ is an atomic concept. We say that {\em $A$ directly
uses $B$} if there is a concept deﬁnition $A \equiv C$ in $\mathcal{T}$ such that $B$ occurs in $C$. We say that {\em $A$ uses $B$} if $A$ directly uses $B$, or if there is a concept
name $B'$ such that $A$ uses $B'$ and $B'$ directly uses $B$.
A finite set $\mathcal{T}$ of concept definitions is an {\em acyclic} TBox if

\noindent 1. there is no concept  in $\mathcal{T}$ that uses itself,

\noindent 2. no concept  occurs more than once on the left-hand side of a concept deﬁnition in $\mathcal{T}$.

\noindent A finite set $\mathcal{T}$ of concept definitions is a {\em completely unravelled} TBox if it is acyclic and no atomic concept occurring on the left-hand side of a concept definition in $\mathcal{T}$ is also present in some other concept definition on the right-side.

\noindent Checking the consistency of a  knowledge base is a key problem in description logics, usually solved via tableaux algorithms. In the $\mathcal{ALC}$ case,  checking the consistency of any  knowledge base is EXPTIME-complete while checking the consistency of a  knowledge base with acyclic TBoxes is PSPACE-complete \cite{DLhandbook}.

\subsection{Basic normal non-distributive modal logic and its semantics}\label{ssec:LE-logic}
 \label{sssec:basic non-distributive modal logic }
 The logic introduced in this section is part of a family of lattice-based  logics, sometimes referred to as {\em LE-logics} (cf.~\cite{conradie2019algorithmic}), which have been studied in the context of a research program on the logical foundations of categorization theory \cite{conradie2016categories,conradie2017toward,conradie2021rough,conradie2020non}.
Let $\Prop$ be a (countable) set of atomic propositions. The language $\mathcal{L}$ is defined as follows:
\[
  \varphi \ceq \bot \mid \top \mid p \mid  \varphi \wedge \varphi \mid \varphi \vee \varphi \mid \Box \varphi \mid \Diamond \varphi,  
\]
where $p\in \Prop$, and $\Box \in \mathcal{G}$ and  $\Diamond \in \mathcal{F}$ for finite sets $\mathcal{F}$ and $\mathcal{G}$ of unary  $\Diamond$-type (resp.~$\Box$-type)  modal operators.
The {\em basic}, or {\em minimal normal} $\mathcal{L}$-{\em logic} is a set $\mathbf{L}$ of sequents $\phi\vdash\psi$,  with $\phi,\psi\in\mathcal{L}$, containing the following axioms for every $\Box \in \mathcal{F}$ and $\Diamond \in \mathcal{G}$:
\smallskip

{{\centering
\begin{tabular}{ccccccccccccc}
     $p \vdash p$ & \qquad & $\bot \vdash p$ & \qquad & $p \vdash p \vee q$ & \qquad & $p \wedge q \vdash p$ & \qquad & $\top \vdash \Box\top$ & \qquad & $\Box p \wedge \Box q \vdash \Box(p \wedge q)$
     \\
     & \qquad & $p \vdash \top$ & \qquad & $q \vdash p \vee q$ & \qquad & $p \wedge q \vdash q$ &\qquad &  $\Diamond\bot \vdash \bot$ & \qquad & $\Diamond(p \vee q) \vdash \Diamond p \vee \Diamond q$\\
\end{tabular}
\par}}
\smallskip
\noindent and closed under the following inference rules:
		{\small{
		\begin{gather*}
			\frac{\phi\vdash \chi\quad \chi\vdash \psi}{\phi\vdash \psi}
			\ \ 
			\frac{\phi\vdash \psi}{\phi\left(\chi/p\right)\vdash\psi\left(\chi/p\right)}
			\ \ 
			\frac{\chi\vdash\phi\quad \chi\vdash\psi}{\chi\vdash \phi\wedge\psi}
			\ \ 
			\frac{\phi\vdash\chi\quad \psi\vdash\chi}{\phi\vee\psi\vdash\chi}
\ \ 
			\frac{\phi\vdash\psi}{\Box \phi\vdash \Box \psi}
\ \ 
\frac{\phi\vdash\psi}{\Diamond \phi\vdash \Diamond \psi}
\end{gather*}}}

Note that unlike in classical modal logic, we cannot assume that  $\Box$ and $\Diamond$ are inter-definable in LE-logics, hence we take all connectives as primitive.

\fakeparagraph{Relational semantics.}
\label{sssec:relsem}
The following notation,  notions and facts are from \cite{conradie2021rough,conradie2020non}.
For any binary relation $T\subseteq U\times V$, and any $U'\subseteq U$  and $V'\subseteq V$, we let $T^c$ denote the set-theoretic complement of $T$ in $U\times V$, and
{\small\begin{equation}\label{eq:def;round brackets}T^{(1)}[U']\ceq\{v\mid \forall u(u\in U'\Rightarrow uTv) \}  \quad\quad T^{(0)}[V']\ceq\{u\mid \forall v(v\in V'\Rightarrow uTv) \}.\end{equation}}
 In what follows, we fix two sets $A$ and $X$, and use $a, b$ (resp.~$x, y$) for elements of $A$ (resp.~$X$), and $B, C, A_j$ (resp.~$Y, W, X_j$) for subsets of $A$ (resp.~of $X$).

A {\em polarity} or {\em formal context} (cf.~\cite{ganter2012formal}) is a tuple $\mathbb{P} =(A,X,I)$, where $A$ and $X$ are sets, and $I \subseteq A \times X$ is a binary relation. 
Intuitively, formal contexts can be understood as abstract representations of databases \cite{ganter2012formal}, so that  $A$ and  $X$ represent collections of {\em objects} and {\em features}, and for any object $a$ and feature $x$, the tuple $(a, x)$ belongs to $I$ exactly when object $a$ has feature $x$.

 As is well known, for every formal context $\mathbb{P} = (A, X, I)$, the pair of maps 

 {{\centering
 $(\cdot)^\uparrow: \mathcal{P}(A)\to \mathcal{P}(X)\quad \mbox{ and } \quad(\cdot)^\downarrow: \mathcal{P}(X)\to \mathcal{P}(A),$
\par}}
\smallskip 
\noindent  defined by the assignments $B^\uparrow \ceq I^{(1)}[B]$ and $Y^\downarrow \ceq I^{(0)}[Y]$,  form a Galois connection, and hence induce the closure operators $(\cdot)^{\uparrow\downarrow}$ and $(\cdot)^{\downarrow\uparrow}$ on $\mathcal{P}(A)$ and on $\mathcal{P}(X)$ respectively. The fixed points of $(\cdot)^{\uparrow\downarrow}$ and $(\cdot)^{\downarrow\uparrow}$ are  the {\em Galois-stable} sets.
A {\em formal concept} of a polarity $\mathbb{P}=(A,X,I)$ is a tuple $c=(B,Y)$ such that $B\subseteq A$ and $Y\subseteq X$, and $B = Y^\downarrow$ and $Y = B^\uparrow$.  The subset $B$ (resp.~$Y$) is   the {\em extension} (resp.~the {\em intension}) of $c$ and is denoted by $\val{c}$  (resp.~$\descr{c}$). 
It is well known (cf.~\cite{ganter2012formal}) that the sets $B$ and $Y$ are  Galois-stable, and
 that the set  of formal concepts of a polarity $\mathbb{P}$,  with the order defined by

{{\centering
$c_1 \leq c_2 \quad \text{iff} \quad \val{c_1} \subseteq \val{c_2} \quad \text{iff} \quad \descr{c_2} \subseteq \descr{c_1}$,
\par
}}
\smallskip
\noindent forms a complete lattice $\mathbb{P}^+$, namely  the {\em concept lattice} of $\mathbb{P}$. 

For the language $\mathcal{L}$ defined above, an {\em enriched formal $\mathcal{L}$-context} is a tuple $\mathbb{F} =(\mathbb{P}, \mathcal{R}_\Box, \mathcal{R}_\Diamond)$, where $\mathcal{R}_\Box= \{R_\Box  \subseteq A \times X \mid \Box \in \mathcal{G}\} $ and $ \mathcal{R}_\Diamond = \{R_\Diamond \subseteq X \times A \mid \Diamond \in \mathcal{F}\}$  are sets of {\em $I$-compatible} relations, that is, for all $\Box \in \mathcal{G}$, $\Diamond \in \mathcal{F}$, $a \in A$, and $x \in X$, the sets $R_\Box^{(0)}[x]$, $R_\Box^{(1)}[a]$, $R_\Diamond^{(0)}[a]$, $R_\Diamond^{(1)}[x]$  are Galois-stable in $\mathbb{P}$. For each $\Box\in \mathcal{G}$ and $\Diamond\in \mathcal{F}$, their associated relations $R_\Box$ and $R_\Diamond$ provide their corresponding semantic interpretations as operations $[R_\Box]$ and $\langle R_\Diamond\rangle$ on the concept lattice $\mathbb{P}^+$ defined as follows:
For any $c \in \mathbb{P}^+$,
{{\centering
      $[R_\Box] c =(R_\Box^{(0)}[\descr{c}], I^{(1)}[R_\Box^{(0)}[\descr{c}]]) \quad \text{and} \quad \langle R_\Diamond\rangle  c =( I^{(0)}[R_\Diamond^{(0)}[\val{c}]], R_\Diamond^{(0)}[\val{c}]).$
\par
}}
\smallskip
We refer to the algebra $\mathbb{F}^+=(\mathbb{P}^+, \{[R_\Box]\}_{\Box\in\mathcal{G}}, \{\langle R_\Diamond \rangle\}_{\Diamond\in\mathcal{F}})$ as the {\em complex algebra} of  $\mathbb{F}$. 

A {\em valuation} on such an $\mathbb{F}$
is a map $V\colon\Prop\to \mathbb{P}^+$. For each  $p\in \Prop$, we let  $\val{p} \ceq \val{V(p)}$ (resp.~$\descr{p}\ceq \descr{V(p)}$) denote the extension (resp.\ intension) of the interpretation of $p$ under $V$.  

A {\em model} is a tuple $\mathbb{M} = (\mathbb{F}, V)$ where $\mathbb{F} = (\mathbb{P}, \mathcal{R}_{\Box}, \mathcal{R}_{\Diamond})$ is an enriched formal context and $V$ is a  valuation on $\mathbb{F}$. For every $\phi\in \mathcal{L}$, we let  $\val{\phi}_\mathbb{M} \ceq \val{V(\phi)}$ (resp.~$\descr{\phi}_\mathbb{M}\ceq \descr{V(\phi)}$) denote the extension (resp.\ intension) of the interpretation of $\phi$ under the homomorphic extension of $V$.   The following `forcing' relations can be recursively defined as follows: 
\smallskip

\smallskip

{{\centering 
\scriptsize
\begin{tabular}{l@{\hspace{1em}}l@{\hspace{2em}}l@{\hspace{1em}}l}
$\mathbb{M}, a \Vdash p$ & iff $a\in \val{p}_{\mathbb{M}}$ &
$\mathbb{M}, x \succ p$ & iff $x\in \descr{p}_{\mathbb{M}}$ \\
$\mathbb{M}, a \Vdash\top$ & always &
$\mathbb{M}, x \succ \top$ & iff   $a I x$ for all $a\in A$\\
$\mathbb{M}, x \succ  \bot$ & always &
$\mathbb{M}, a \Vdash \bot $ & iff $a I x$ for all $x\in X$\\
$\mathbb{M}, a \Vdash \phi\wedge \psi$ & iff $\mathbb{M}, a \Vdash \phi$ and $\mathbb{M}, a \Vdash  \psi$ & 
$\mathbb{M}, x \succ \phi\wedge \psi$ & iff $(\forall a\in A)$ $(\mathbb{M}, a \Vdash \phi\wedge \psi \Rightarrow a I x)$
\\
$\mathbb{M}, x \succ \phi\vee \psi$ & iff  $\mathbb{M}, x \succ \phi$ and $\mathbb{M}, x \succ  \psi$ & 
$\mathbb{M}, a \Vdash \phi\vee \psi$ & iff $(\forall x\in X)$ $(\mathbb{M}, x \succ \phi\vee \psi \Rightarrow a I x)$.
\end{tabular}
\par}}
\smallskip

\noindent As to the interpretation of modal formulas, for every $\Box\in\mathcal{G}$ and $\Diamond\in\mathcal{F}$:
\smallskip

{{\centering
\scriptsize
\begin{tabular}{llcll}
$\mathbb{M}, a \Vdash \Box\phi$ &  iff $(\forall x\in X)(\mathbb{M}, x \succ \phi \Rightarrow a R_\Box x)$ & \quad\quad &
$\mathbb{M}, x \succ \Box\phi$ &  iff $(\forall a\in A)(\mathbb{M}, a \Vdash \Box\phi \Rightarrow a I x)$\\
$\mathbb{M}, x \succ \Diamond\phi$ &  iff for all $ a\in A$, if $\mathbb{M}, a \Vdash \phi$ then $x R_\Diamond a$ &&
$\mathbb{M}, a \Vdash \Diamond\phi$ & iff $(\forall x\in X)(\mathbb{M}, x \succ \Diamond\phi \Rightarrow a I x)$  \\
\end{tabular}
\par}}
\smallskip

\noindent The definition above ensures that, for any $\mathcal{L}$-formula $\varphi$,

{{\small\centering
$\mathbb{M}, a \Vdash \phi$  iff  $a\in \val{\phi}_{\mathbb{M}}$, \quad  and \quad$\mathbb{M},x \succ \phi$  iff  $x\in \descr{\phi}_{\mathbb{M}}$. \par}}

{{\small\centering
$\mathbb{M}\models \phi\vdash \psi$ \quad iff \quad $\val{\phi}_{\mathbb{M}}\subseteq \val{\psi}_{\mathbb{M}}$\quad  iff  \quad  $\descr{\psi}_{\mathbb{M}}\subseteq \descr{\phi}_{\mathbb{M}}$. 
\par}}

The interpretation of the propositional connectives $\vee$ and $\wedge$ in  the framework described above reproduces the standard notion of join and the meet of formal concepts used in FCA. The interpretation of the operators $\Box$ and $\Diamond$  is motivated by algebraic properties and  duality theory for modal operators on lattices (cf.~\cite[Section 3]{conradie2020non} for an expanded discussion). In \cite[Proposition 3.7]{conradie2021rough}, it is shown that the semantics of LE-logics is compatible with Kripke semantics for classical modal logic, and thus, LE-logics are indeed generalizations of classical modal logic. This interpretation is further justified in \cite[Section 4]{conradie2021rough} by noticing  that, under 
the interpretations of the relation $I$ as 
$a I x$ iff ``object $a $  has feature $x$''
and  $R=R_\Box =R^{-1}_\Diamond$ as $a R x$ iff ``there is evidence that object $a$  has feature $x$'', then, for any concept $c$, the extents of concepts $\Box c$ and $\Diamond c$ can be interpreted as ``the set of objects which {\em certainly} belong to $c$'' (upper approximation), and ``the set of objects which {\em possibly} belong to $c$'' (lower approximation) respectively. Thus, the interpretations of $\Box$ and $\Diamond$ have similar meaning in the LE-logic as in the classical modal logic. A similar justification regarding similarity of epistemic interpretations of $\Box$ in classical and lattice-based modal logics  is discussed in \cite{conradie2017toward}.  This transfer of meaning of modal axioms from classical modal logic to LE-logics has been investigated  as a general phenomenon in \cite[Section 4.3]{conradie2022modal}, \cite{conradie2020non}.

\section{LE Description logic}
\label{Sec: LE-DL}
In this section, we  introduce the  non-classical  DL LE-$\mathcal{ALC}$, so that LE-$\mathcal{ALC}$ will be in same relation with LE-logic as $\mathcal{ALC}$ is  with classical modal logic. This similarity extends to the models we will introduce for LE-$\mathcal{ALC}$: in the same way as Kripke models of classical modal logic are used as models of  $\mathcal{ALC}$,  enriched formal contexts, which provide complete semantics for   LE-logic, will serve as models of LE-$\mathcal{ALC}$. 
In this specific respect,  LE-$\mathcal{ALC}$ can be seen as a  generalization  of the positive fragment (i.e.~the fragment with no negations in concepts) of $\mathcal{ALC}$ in which we do not assume distributivity laws to hold for concepts. Consequently, the language of LE-$\mathcal{ALC}$  contains individuals of two  types, usually interpreted as the {\em objects} and {\em features} of the given database or categorization. Let $\mathsf{OBJ}$ and $\mathsf{FEAT}$ be  disjoint sets of individual names for objects and  features. 

The set $\mathcal{R}$ of the role names for LE-$\mathcal{ALC}$ is the  union of three disjoint sets of  relations: (1) the single  relation $I \subseteq \mathsf{OBJ} \times \mathsf{FEAT} $   (2) a set  $\mathcal{R}_\Box$ such that for any $R_\Box \in \mathcal{R}_\Box$m, $R_\Box \subseteq  \mathsf{OBJ} \times \mathsf{FEAT}$; (3) a set  $\mathcal{R}_\Diamond$ such that for any $R_\Diamond \in \mathcal{R}_\Box$m, $R_\Diamond \subseteq  \mathsf{FEAT} \times \mathsf{OBJ}$. While $I$ is intended to be interpreted as the incidence relation of formal concepts, and  encodes information on which objects have which features, the relations in $\mathcal{R}_\Box$  and $\mathcal{R}_\Diamond$
encode additional relationships between  objects and  features (cf.~\cite{conradie2021rough} for an extended discussion).

For any set  $\mathcal{C}$
 of atomic concept names, the  language of LE-$\mathcal{ALC}$ concepts  is:

{{\centering
 $C \ceq D\ |\ C_1 \wedge C_2\ |\ C_1\vee C_2\  |\ \langle R_\Diamond \rangle C\ |\ [R_\Box ]C$   
\par}}

\noindent where $D \in \mathcal{C}$, $R_\Box \in \mathcal{R}_\Box$ and $R_\Diamond \in \mathcal{R}_\Diamond$. 
This language  matches the language of LE-logic, and has an analogous intended  interpretation on the complex algebras of enriched formal contexts  (cf.~Section \ref{sssec:basic non-distributive modal logic }).
 As usual, $\vee$ and $\wedge$ are to be interpreted as the smallest common superconcept  and the greatest common subconcept as in FCA. The constants $\top$ and $\bot$ are to be interpreted as the largest and the smallest concept, respectively. We do not include $\neg C$ as a valid concept  in our language, since there is no canonical and natural way to interpret negations in non-distributive settings.  
 
 The concepts $\langle R_\Diamond \rangle C$ and $[R_\Box]C $  in  LE-$\mathcal{ALC}$ are intended to be interpreted as the operations $\langle R_\Diamond\rangle$ and  $[R_\Box]$ defined by the interpretations of their corresponding role names  in enriched formal contexts, analogously to the way in which $\exists r$ and $\forall r$ in $\mathcal{ALC}$ are interpreted on Kripke frames. 
 We do not use the symbols $\forall r$ and $\exists r$ in the context of LE-$\mathcal{ALC}$ because, as discussed in Section \ref{sssec:basic non-distributive modal logic }, the semantic clauses of modal operators in LE-logic use universal quantifiers, and hence using the same notation verbatim would  be ambiguous or misleading.

TBox assertions in LE-$\mathcal{ALC}$   are of the  shape $C_1 \equiv C_2$, where $C_1$ and $C_2$ are concepts defined as above.\footnote{As is standard in DL (cf.~\cite{DLhandbook} for more details), general concept inclusion of the form $C_1 \sqsubseteq C_2$ can be rewritten as concept definition $C_1 \equiv C_2 \wedge C_3$, where $C_3$ is a new atomic concept name.} The ABox assertions are of the form:
\noindent

{{\centering
  $aR_\Box x,\quad xR_\Diamond a,\quad aIx,\quad a:C,\quad x::C,\quad \neg \alpha,$  
\par}}

\noindent
where $\alpha$ is any of the first five ABox terms, and $C$ is any concept in the language of LE-$\mathcal{ALC}$. We refer to the terms of first three types as {\em relational terms}. The interpretations of the terms $a:C$ and  $x::C$ are: ``object $a$ is a member of concept $C$", and  ``feature  $x$ is in the description of  concept $C$", respectively.

An {\em interpretation} for LE-$\mathcal{ALC}$ is a tuple $\mathrm{I} = (\mathbb{F}, \cdot^\mathrm{I}) $, where $\mathbb{F}=(\mathbb{P}, \mathcal{R}_\Box, \mathcal{R}_\Diamond)$ is
an enriched formal context, and $\cdot^\mathrm{I}$ maps:

\noindent 1. individual names $a \in \mathsf{OBJ}$ (resp.~$x \in \mathsf{FEAT}$),   to some $a^\mathrm{I} \in A$ (resp.~$x^\mathrm{I} \in X$);

\noindent 2. relation names $I$, $R_\Box$ and $R_\Diamond$ to  relations $I^\mathrm{I}$, $R_\Box^\mathrm{I}$ and $R_\Diamond^\mathrm{I}$ in $\mathbb{F}$;

\noindent 3. any atomic concept name $D$ to  $D^\mathrm{I} \in \mathbb{F}^+$, and other concepts as follows:
\smallskip

{{\centering
    \begin{tabular}{l l }

    $(C_1 \wedge C_2)^{\mathrm{I}} = C_1^{\mathrm{I}} \wedge C_2^{\mathrm{I}}$\qquad   & $(C_1\vee C_2)^{\mathrm{I}} = C_1^{\mathrm{I}} \vee C_2^{\mathrm{I}}$\\  
    $([R_\Box]C)^\mathrm{I} = [R_\Box^\mathrm{I}]C^{\mathrm{I}}$ \qquad  &  $(\langle R_\Diamond \rangle C)^\mathrm{I} =\langle  R_\Diamond^{\mathrm I} \rangle C^{\mathrm{I}} $\\
    \end{tabular}
\par}}
\smallskip
where the operators $[R^\mathrm{I}_\Box ]$ and $\langle R^\mathrm{I}_\Diamond \rangle$ are defined as in Section \ref{sssec:basic non-distributive modal logic }.

The satisfiability relation for an interpretation $\mathrm{I}$ is defined as follows:

\noindent 1. $\mathrm{I} \models C_1\equiv C_2$ iff $\val{C_1^\mathrm{I}} = \val{C_2^\mathrm{I}}$ iff $\descr{C_2^\mathrm{I}} = \descr{C_1^\mathrm{I}}$.

\noindent 2. $\mathrm{I} \models a:C$ iff $a^\mathrm{I} \in \val{C^\mathrm{I}}$
and  $\mathrm{I} \models x::C$ iff $x^\mathrm{I} \in \descr{C^\mathrm{I}}$.

\noindent 3. $ \mathrm{I} \models  a I x$ (resp.~$a R_\Box x$, $x R_\Diamond a$) iff $a^{\mathrm{I}} I^{\mathrm{I}} x^{\mathrm{I}} $ (resp.~$a^{\mathrm{I}} R_\Box^{\mathrm{I}} x^{\mathrm{I}} $, $x^{\mathrm{I}} R_\Diamond^{\mathrm{I}} a^{\mathrm{I}} $). 

\noindent 4. $ \mathrm{I} \models  \neg \alpha$,  where $\alpha$ is any ABox term, iff $ \mathrm{I} \not\models  \alpha$. 

\noindent An interpretation $\mathrm{I}$ is a {\em model } for an  LE-$\mathcal{ALC}$ knowledge base $(\mathcal{A}, \mathcal{T})$
if  $\mathrm{I}\models \mathcal{A}$ and $\mathrm{I} \models \mathcal{T}$. An LE-$\mathcal{ALC}$ knowledge base $(\mathcal{A}, \mathcal{T})$ (resp.~TBox $\mathcal{T}$, resp.~ABox $\mathcal{A}$) is said to be {\em inconsistent} if it has no model.

The framework of LE-$\mathcal{ALC}$ formally brings  FCA and DL together in two important ways: (1) the concepts of LE-$\mathcal{ALC}$   are naturally interpreted as formal concepts in FCA;  (2) the language of LE-$\mathcal{ALC}$  is designed  to represent knowledge and reasoning in the setting of  enriched formal contexts.

\section{Tableaux algorithm for ABox of LE-$\mathcal{ALC}$}
\label{Sec: tableau}
In this section, we define a tableaux algorithm for checking the consistency of LE-$\mathcal{ALC}$ ABoxes. An LE-$\mathcal{ALC}$ ABox $\mathcal{A}$ contains a {\em clash} iff it contains both $\beta$ and $\neg \beta$ for some relational term $\beta$. The expansion rules below are designed so that 
the expansion of $\mathcal{A}$ will contain a clash iff  $\mathcal{A}$ is inconsistent. 
The set   $sub(C)$ of sub-formulas of 
  any  LE-$\mathcal{ALC}$   concept  $C$  is defined  as 
  usual.

A concept  $C'$ {\em occurs}  in $\mathcal{A}$ (in symbols: $C' \in \mathcal{A}$) if $C'\in sub(C)$ for some $C$ such that  one of the terms $a:C$,  $x::C$, $\neg a:C$, or $\neg x ::C$ is in $\mathcal{A}$. A constant $b$  (resp.~$y$) {\em occurs} in $\mathcal{A}$ ($b \in \mathcal{A}$, or $y \in \mathcal{A}$), iff some term containing $b$ (resp.~$y$) occurs in it. 

The tableaux algorithm below constructs a model 
$(\mathbb{F},\cdot^\mathrm{I})$ for every consistent $\mathcal{A}$, where 
$\mathbb{F}= (\mathbb{P}, \mathcal{R}_\Box, \mathcal{R}_\Diamond)$ is such that,
for  any $C \in \mathcal{A}$, some   $a_C \in A$ and $x_C \in X$ exist such that, for any  $a \in A$ (resp.~any $x \in X$), $a \in \val{C^{\mathrm{I}}}$ (resp.~$x \in \descr {C}^{\mathrm{I}}$) iff $a I x_C$ (resp.~$a_C I x$).
We call $a_C$ and $x_C$ the  {\em classifying object} and the {\em classifying feature} of $C$, respectively. To make our notation more easily readable, we will write $a_{\Box C}$, $x_{\Box C}$ (resp.~$a_{\Diamond C}$, $x_{\Diamond C}$) instead of $a_{[R_\Box]C}$, $x_{[R_\Box]C}$ (resp.~$a_{\langle R_\Diamond\rangle C}$, $x_{\langle R_\Diamond\rangle C}$) Moreover, for every $R_\Box \in \mathcal{R}_\Box$ and $R_\Diamond \in \mathcal{R}_\Diamond$, we will also impose the condition that  $a \in \val{[R_\Box]C}$ (resp.~$x \in \descr{\langle R_\Diamond \rangle C}$)  iff $a R_\Box x_C$ (resp.~$x R_\Diamond a_C$), where $a_C$ and $x_C$ are the classifying object and the classifying feature of $C$, respectively. Note that we can always assume w.l.o.g.~that any consistent  ABox $\mathcal{A}$ is satisfiable in  a model with classifying objects and features (cf.~Theorem \ref{thm:completeness}).

%
%
\vspace{-0.2 cm}
\begin{algorithm} 
\caption{tableaux algorithm for checking LE-$\mathcal{ALC}$ ABox consistency }\label{alg:main algo}
\label{alg:tableaux}
    \hspace*{\algorithmicindent} \textbf{Input}: An   LE-$\mathcal{ALC}$ ABox $\mathcal{A}$. \quad \textbf{Output}: whether $\mathcal{A}$ is inconsistent. 
    \begin{algorithmic}[1]
        \State \textbf{if} there is a clash in $\mathcal{A}$ \textbf{then} \textbf{return} ``inconsistent''.
        \State \textbf{pick} any applicable expansion rule $R$, \textbf{apply} $R$ to $\mathcal{A}$ and proceed recursively.
      \State \textbf{if}  no expansion  rule is applicable   \textbf{return} ``consistent''.
    \end{algorithmic}
\end{algorithm}
\vspace{-0.4 cm}

\noindent Below, we list  the expansion rules. The commas in  each rule are  metalinguistic conjunctions, hence every tableau is non-branching. 
\smallskip

{{\footnotesize
\centering
\begin{tabular}{cccc} 
\mc{1}{c}{\textbf{Creation rule}} & \mc{1}{c}{\textbf{Basic rule}} \\
\AXC{For any $C \in \mathcal{A}$}
\RL{\fns create}
\UIC{$a_C:C$, \quad $x_C::C$}
\DP
 \ & \ 
\rule[-1.85mm]{0mm}{8mm}
\AXC{$b:C, \quad y::C$}
\LL{\fns $I$}
\UIC{$b I y$}
\DP 
\ & \
\\


\end{tabular}

\begin{tabular}{cccc}
\mc{2}{c}{\textbf{Rules for the logical connectives}} & \mc{2}{c}{\textbf{$I$-compatibility rules}} \\
\rule[-1.85mm]{0mm}{8mm}
 \AXC{$b:C_1 \wedge  C_2$}
\RL{\fns $\wedge_A$}
\UIC{$b:C_1,$ \quad $b:C_2$}
\DP 
\ & \

\AXC{$y::C_1 \vee  C_2$}
\LL{\fns $\vee_X$}
\UIC{$y::C_1,$ \quad $y::C_2$}
\DP 
\  &  \

\AXC{$b I \Box y$}
\LL{\fns $\Box y$}
\UIC {$b R_\Box y$}
\DP

\ & \

\AXC{$b I \blacksquare y$}
\RL{\fns $\blacksquare y$}
\UIC {$y R_\Diamond b$}
\DP
\\[3mm]

 \AXC{$b:[R_\Box]C,$ \quad  $y::C$}
\LL{\fns $\Box$}
\UIC{$b R_\Box y$}
\DP 
\ & \ 
  \AXC{$y::\langle  R_\Diamond \rangle C,$ \quad  $b:C$}
\RL{\fns $\Diamond$}
\UIC{$y R_\Diamond b$}
\DP 
\ & \
\AXC{$\Diamond b I  y$}
\LL{\fns $\Diamond b$}
\UIC {$y R_\Diamond b$}
\DP
\ & \
\AXC{$\Diamondblack b I y$}
\RL{\fns $\Diamondblack b$}
\UIC {$b R_\Box y$}
\DP
\\[3 mm]

\end{tabular}
\begin{tabular}{rl}
\mc{2}{c}{\textbf{inverse rules for the lattice connectives}}  \\

   \AXC{$b:C_1$, $b:C_2$, $C_1 \wedge C_2 \in \overline{\mathcal{A}}$}
\LL{\fns $\wedge_A^{-1}$}
\UIC {$b:C_1 \wedge C_2$}
\DP
\ & \
\AXC{$y::C_1$, $y::C_2$, $C_1 \vee C_2 \in \overline{\mathcal{A}}$}
\RL{\fns $\vee_X^{-1}$}
\UIC {$y::C_1 \vee C_2$}
\DP
\end{tabular}

\begin{tabular}{rl}
\mc{2}{c}{\textbf{Adjunction rules}}  \\
\rule[-1.85mm]{0mm}{8mm}
\AXC{$ b R_\Box y$}
\LL{\fns $R_\Box$}
\UIC{$\Diamondblack b I y,$ \quad  $b I \Box y$}
\DP 
\ & \ 
\AXC{$y R_\Diamond b$}
\RL{\fns $R_\Diamond$}
\UIC{$\Diamond b I y,$ \quad  $b I \blacksquare y$}
\DP 
\\ [2mm]
\end{tabular}

\begin{tabular}{cccc}
\mc{2}{c}{\textbf{Basic rules for negative assertions}} & \mc{2}{c}{\textbf{Appending  rules}} \\
\rule[-1.85mm]{0mm}{8mm}
\AXC{$\neg (b:C)$}
\LL{\fns $\neg b$}
\UIC{$\neg (b I x_C)$}
\DP 
\ & \ 
\AXC{$\neg (x::C)$}
\RL{\fns $\neg x$}
\UIC{$\neg (a_C I x)$}
\DP 
\ & \ 
\AXC{$b I x_C$}
\LL{\fns $x_C$}
\UIC{$b:C$}
\DP 
\ & \ 
\AXC{$a_C I y$}
\RL{\fns $a_C$}
\UIC{$y::C$}
\DP 
\\
\end{tabular}
\par}}
\smallskip 

\noindent In the adjunction rules the individuals $\Diamondblack b$, $\Diamond b$, $\Box y$,   and $\blacksquare y$ are new and unique for each relation $R_\Box$ and $R_\Diamond$, except for  $\Diamond a_C= a_{\Diamond C}$ and 
$\Box x_C= x_{\Box C}$. Side conditions for rules $\wedge_A^{-1}$, and $\vee_X^{-1}$ ensure we do not add new  joins or meets to concepts.

It is easy to check that the following rules are derivable in the calculus. 
\begin{tabular}{cc}
\AXC{$b:C_1 \vee C_2,$ \quad  $y::C_1,$ \quad $y::C_2$}
\LL{\fns $\vee_A$}
\UIC{$b I y$}
\DP 
\ & \  

 \AXC{$y::C_1 \wedge  C_2,$ \quad  $b:C_1,$ \quad $b:C_2$}
\RL{\fns $\wedge_X$}
\UIC{$b I y$}
\DP 
\\[3mm]
\AXC{$\Diamondblack b:C$}
\LL{\fns $adj_\Box$}
\UIC{$b: [R_\Box] C$}
\DP 
\ & \ 
\AXC{$\blacksquare y::C$}
\RL{\fns $adj_\Diamond$}
\UIC{$y::\langle R_\Diamond \rangle  C$}
\DP 
\\
\end{tabular}

The basic rule and the logical rules for the connectives encode the semantics of the logical connectives in LE-$\mathcal{ALC}$. The creation rule makes sure that, whenever successful, the algorithm outputs models with classifying object $a_C$ and feature $x_C$ for every concept  $C \in \mathcal{A}$. The adjunction rules and $I$-compatibility rules imply that every $R_\Box \in \mathcal{R}_\Box$ and $R_\Diamond \in \mathcal{R}_\Diamond$ are $I$-compatible. Appending and negative assertion rules encode the defining property of classifying objects and features of concepts. 

\begin{example}
Let $\mathcal{A} =\{b:[R_\Box][R_\Box] C_1, b:[R_\Box][R_\Box]C_2, y::[R_\Box] (C_1 \wedge C_2), \neg (b R_\Box y)\}$. It is easy to check that $\mathcal{A}$ has no LE-$\mathcal{ALC}$ model. The algorithm applies on $\mathcal{A}$ as follows (we only do the partial expansion to show that the clash exists):
\smallskip

{{\centering\footnotesize
\begin{tabular}{|c|c|c|l}
\cline{1-3}
    Rule & Premises & Added terms   \\
    \cline{1-3}
     Creation &      & $x_{\Box C_1}\!::\![R_\Box] C_1$, $x_{\Box C_2}\!::\![R_\Box] C_2$, $x_{C_1 \wedge C_2}::C_1 \wedge C_2$ \\
     $\Box$ & $x_{\Box C_i}\!::\![R_\Box] C_i$, $b\!:\![R_\Box][R_\Box] C_i$ & $b R_\Box x_{\Box C_i}$ & $i=1,2$\\
     $R_\Box$  &$b R_\Box x_{\Box C_i}$&   $\Diamondblack b I x_{\Box C_i}$ & $i=1,2$ \\
     Appending & $\Diamondblack b I x_{\Box C_i}$&  $\Diamondblack b:[R_\Box] C_i$ & $i=1,2$\\
     \cline{1-3}
\end{tabular}
\par}}
\smallskip\noindent
By applying the same process to $\Diamondblack b:[R_\Box] C_1$, $\Diamondblack b:[R_\Box] C_2$ and $x_{\Box C_1}::[R_\Box] C_1$, $x_{\Box C_2}::[R_\Box] C_2$, we add the terms $\Diamondblack \Diamondblack b:C_1$ and $\Diamondblack \Diamondblack b:C_2$ to the tableau. Then the further  tableau expansion is as follows: 

\smallskip
{{\centering 
\begin{tabular}{|c|c|c|l}
\cline{1-3}
Rule & Premises & Added terms & \\
\cline{1-3}
$\wedge_X$ &  $x_{C_1 \wedge C_2}::C_1 \wedge C_2$, $\Diamondblack\Diamondblack b:C_1$, $\Diamondblack\Diamondblack b:C_2$,  $\Diamondblack \Diamondblack b:C_1$ & $\Diamondblack \Diamondblack  b I x_{C_1 \wedge C_2}$\\
Appending & $\Diamondblack \Diamondblack  b I x_{C_1 \wedge C_2}$& $\Diamondblack \Diamondblack  b : C_1 \wedge C_2$ \\
$adj_\Box$ (twice)  & $\Diamondblack \Diamondblack  b : C_1 \wedge C_2$  & $b:[R_\Box][R_\Box](C_1 \wedge C_2)$\\
$\Box$ &  $b:[R_\Box][R_\Box] (C_1 \wedge C_2)$, $y::[R_\Box](C_1 \wedge C_2)$ & $b R_\Box y$\\
\cline{1-3}
\end{tabular}
\par}}
\smallskip
\noindent Thus, there is a clash between  $\neg (b R_\Box y)$ and $b R_\Box y$ in the expansion. 
\end{example}

\begin{example}
Let $\mathcal{A} = \{\neg(b I y), y::C_1, \neg (b:C_2), b:C_1 \vee C_2, b R_\Box y\}$.  The following table shows the tableau expansion for $\mathcal{A}$. Let $\mathcal{W} \coloneqq \{C_1, C_2, C_1 \vee C_2\}$.
\smallskip

{{\centering\footnotesize
\begin{tabular}{|c|c|c|}
\hline
     Rule & Premises & Added terms  \\
     \hline
       Creation  & & $a_C:C$, $x_C::C$, $C \in \mathcal{W}$\\
       Basic & $a_C:C$, $x_C::C$, $C \in \mathcal{W}$  & $a_C I x_C$, $C \in \mathcal{W}$\\ 
        Appending & $a_{C_1} I x_{C_1 \vee C_2}$, $a_{C_2} I x_{C_1 \vee C_2}$& $a_{C_1}:C_1 \vee C_2$,  $a_{C_2}:C_1 \vee C_2$\\
       $\vee_X$ & $x_{C_1 \vee C_2}::C_1 \vee C_2$ & $x_{C_1 \vee C_2}::C_1$, $x_{C_1 \vee C_2}::C_2$\\
       Basic & $a_{C_1}::C_1 \vee C_2$, $x_{C_1 \vee C_2}::C_1$ & $a_{C_1} I x_{C_1 \vee C_2}$\\
        Basic & $a_{C_2}::C_1 \vee C_2$, $x_{C_1 \vee C_2}::C_1$ & $a_{C_2} I x_{C_1 \vee C_2}$\\
        $R_\Box$ & $b R_\Box y$ & $\Diamondblack b I y$, $b I \Box y$\\
        $\neg_b$ & $\neg(b:C_2)$&  $\neg (b I x_{C_2})$\\
        \hline
\end{tabular}
\par}}
\smallskip
 Note that no expansion rule is applicable anymore. It is clear that the tableau does not contain any clashes. Thus, this ABox has a model. By the procedure described in Section \ref{ssec:Soundness}, this model is given by $\mathcal{R}_\Box = \{R_\Box\}, \mathcal{R}_\Diamond = \{R_\Diamond\}, A =\{a_{C_1}, a_{C_2}, a_{C_1 \vee C_2}, b, \Diamondblack b\}$, $X =\{x_{C_1}, x_{C_2}, x_{C_1 \vee C_2}, y, \Box y \}$, $I=\{(a_{C},x_{C})_{C \in \mathcal{W}},\\ (a_{C_1}, x_{C_1 \vee C_2}), (a_{C_2}, x_{C_1 \vee C_2}), (\Diamondblack b , y), (b, \Box y)\} $, $R_\Box =\{(b,y)\}$, $R_\Diamond = \varnothing$. 
\end{example}

\subsection{Termination of the tableaux algorithm}\label{ssec:Termination of the tableaux}
In this section, we show that  Algorithm~\ref{alg:tableaux} always terminates for any finite LE-$\mathcal{ALC}$ ABox $\mathcal{A}$. 
Since no rule branches, we only  need to check that the number of new individuals added by the expansion rules is finite. Note that the only rules for adding new individuals are the creation and adjunction rules. The creation rules add one new object and feature for every concept $C$ occurring  in the expansion of $\mathcal{A}$. Thus,  it is enough to show that the number of individuals and new concepts added by applying adjunction rules is finite. To do so, we will show that any individual constant introduced by means of any adjunction rule  will contain only finitely many modal operators applied to a constant occurring in $\mathcal{A}$ or added by the creation rule and any new concept  added will contain finitely many $\Box$ and $\Diamond$ operators applied to a concept occurring in $\mathcal{A}$. 
\begin{definition}\label{def:concept depth}
    The   $\Diamond$-depth $\Diamond_{\mathcal{D}}$  and $\Box$-depth $\Box_{\mathcal{D}}$ of  $C$ is defined as follows:
        
    \noindent 1.~If $C$  is an atomic  concept,  then $\Diamond_{\mathcal{D}}(C)=\Box_{\mathcal{D}}(C)=0$;
        
      \noindent 2.~$\Diamond_{\mathcal{D}}(\langle R_\Diamond \rangle C)=\Diamond_{\mathcal{D}}( C)+1 $ and $\Box_{\mathcal{D}}(\langle R_\Diamond \rangle C)=\Box_{\mathcal{D}}(C) $; 
      
      \noindent 3.~$\Diamond_{\mathcal{D}}([R_\Box] C)=\Diamond_{\mathcal{D}}(C) $ and $\Box_{\mathcal{D}}([R_\Box] C)=\Box_{\mathcal{D}}(C)+1 $;
      
    \noindent 4.~$\Diamond_{\mathcal{D}}(C_1 \vee C_2)= \max ( \Diamond_{\mathcal{D}}(C_1), \Diamond_{\mathcal{D}}(C_2))$ and $\Box_{\mathcal{D}}(C_1 \wedge C_2) = \max( \Box_{\mathcal{D}}(C_1), \Box_{\mathcal{D}}(C_2))$; 

     \noindent 5.~$\Diamond_{\mathcal{D}}(C_1 \wedge C_2)= \min ( \Diamond_{\mathcal{D}}(C_1), \Diamond_{\mathcal{D}}(C_2))$ and $\Box_{\mathcal{D}}(C_1 \vee C_2) = \min( \Box_{\mathcal{D}}(C_1), \Box_{\mathcal{D}}(C_2))$.
        
\end{definition}

\begin{definition}\label{def:Constant depth}
  The $\Box$-depth $\Box_{\mathcal{D}}$ and $\Diamond$-depth $\Diamond_{\mathcal{D}}$  of any constants $b$ and $y$ are:
  
        \noindent 1. if $b, y\in \mathcal{A}$,  {\small $\Box_{\mathcal{D}}(b) = \Diamond_{\mathcal{D}}(b)=\Box_{\mathcal{D}}(y) = \Diamond_{\mathcal{D}}(y)=0 $};

         \noindent  2. {\small $\Box_{\mathcal{D}}(a_C)=0$, $
         \Diamond_{\mathcal{D}}(x_C) =0$, 
         $\Box_{\mathcal{D}}(x_C)=-\Box_{\mathcal{D}}(C)$,  
         and $\Diamond_{\mathcal{D}}(a_C)=-\Diamond_{\mathcal{D}}(C)$}; 
        
         \noindent  3.
        {\small $\Box_{\mathcal{D}}(\Diamondblack b)=\Box_{\mathcal{D}}(b)+1 $},  {\small $\Box_{\mathcal{D}}(\Diamond b)=\Box_{\mathcal{D}}(b) $},
          {\small $\Diamond_{\mathcal{D}}(\Diamondblack b)=\Diamond_{\mathcal{D}}(b) $},  {\small $\Diamond _{\mathcal{D}}(\Diamond b)=\Diamond _{\mathcal{D}}(b)-1$}; 
          
          \noindent 4. {\small $\Box_{\mathcal{D}}(\Box y)=\Box_{\mathcal{D}}(y)-1$},  {\small $\Diamond_{\mathcal{D}}(\Box y )=\Diamond_{\mathcal{D}}(y) $}, 
       {\small $\Box_{\mathcal{D}}(\blacksquare y)=\Box_{\mathcal{D}}(y) $},   {\small $\Diamond _{\mathcal{D}}(\blacksquare y)=\Diamond _{\mathcal{D}}( y) +1$}. 
\end{definition}

 \begin{definition}
The {\em$\Box$-depth} (resp.~{\em $\Diamond$-depth}) of an ABox $\mathcal{A}$ is \\$\Box_{\mathcal{D}}(\mathcal{A}) \coloneqq \max \{ \Box_{\mathcal{D}}(C')\ |\ C' \in \mathcal{A}$\} (resp.~$\Diamond_{\mathcal{D}}(\mathcal{A}) \coloneqq \max \{ \Diamond_{\mathcal{D}}(C')\ |\ C' \in \mathcal{A}$\}). 
\end{definition}

The following Lemma bounds the length of concept and individual names appearing in a tableaux. 

\begin{lemma}\label{lem:depth}
For any individual names $b$, and $y$, and concept $C$  added during tableau expansion of ${\mathcal{A}}$, 
\begin{equation}\label{eq:IH 1}
 \Box_{\mathcal{D}}(C) \leq  \Box_{\mathcal{D}}(\mathcal{A})+1 \, 
 \mbox{ and } \, \Diamond_{\mathcal{D}}(C) \leq  \Diamond_{\mathcal{D}}(\mathcal{A})+1,  
\end{equation}

\begin{equation}\label{eq:IH 2}
-\Diamond_{\mathcal{D}}(\mathcal{A} )-1 \leq \Diamond_{\mathcal{D}} (b) \, \mbox{ and } \,   \Box_{\mathcal{D}}(b) \leq  \Box_{\mathcal{D}}(\mathcal{A})+1, 
\end{equation}

\begin{equation}\label{eq:IH 3}
-\Box_{\mathcal{D}}(\mathcal{A})-1 \leq \Box_{\mathcal{D}}(y)\, \mbox{ and } \, \Diamond_{\mathcal{D}}(y) \leq  \Diamond_{\mathcal{D}}(\mathcal{A})+1
\end{equation}

\end{lemma}
 
\begin{proof}

We prove a stronger property of $ \overline{\mathcal{A}}$, obtained from $\mathcal{A}$ after any finite number of expansion steps.

1.~For any term $b I y \in \overline{\mathcal{A}}$, $\Box_{\mathcal{D}} (b) -  \Box_{\mathcal{D}} (y) \leq \Box_{\mathcal{D}}(\mathcal{A})+1$, and $ \Diamond_{\mathcal{D}} (y) -  \Diamond_{\mathcal{D}} (b) \leq \Diamond_{\mathcal{D}}(\mathcal{A})+1$.

2.~For any term $b R_\Box y \in \overline{\mathcal{A}}$, $\Box_{\mathcal{D}} (b) +1 -  \Box_{\mathcal{D}} (y) \leq \Box_{\mathcal{D}}(\mathcal{A})+1$, and $ \Diamond_{\mathcal{D}} (y) -  \Diamond_{\mathcal{D}} (b) \leq \Diamond_{\mathcal{D}}(\mathcal{A})+1$.

3.~For any term $y R_\Diamond b \in \overline{\mathcal{A}}$, $\Box_{\mathcal{D}} (b) -  \Box_{\mathcal{D}} (y) \leq \Box_{\mathcal{D}}(\mathcal{A})+1$, and $ \Diamond_{\mathcal{D}} (y) +1  -  \Diamond_{\mathcal{D}} (b) \leq \Diamond_{\mathcal{D}}(\mathcal{A})+1$.

4.~For any term $b:C \in \overline{\mathcal{A}}$, $ \Box_{\mathcal{D}} (b)  +  \Box_{\mathcal{D}} (C) \leq \Box_{\mathcal{D}}(\mathcal{A})+1$, and $ -\Diamond_{\mathcal{D}} (b) -  \Diamond_{\mathcal{D}} (C) \leq 0$.

5.~For any term $y::C \in \overline{\mathcal{A}}$, $ -\Box_{\mathcal{D}} (y)  - \Box_{\mathcal{D}} (C) \leq 0$, and $ \Diamond_{\mathcal{D}} (y) +  \Diamond_{\mathcal{D}} (C) \leq \Diamond_{\mathcal{D}}(\mathcal{A})+1$.

The proof is by simultaneous (over all the claims) induction on the number of expansion rules applied.
The base case is immediate from the definitions as every individual name in the original ABox has $\Box$-depth and $\Diamond$-depth zero,  and $\Box$-depth and $\Diamond$-depth are positive and bounded by $ \Box_{\mathcal{D}}(\mathcal{A})$ and $\Diamond_{\mathcal{D}}(\mathcal{A})$, respectively. 
 The inductive steps are proved below:

\textbf{creation rule:} it follows immediately from the definitions. 

\textbf{Basic  rule $I$:} in this case the new term $bIy$ is added from the terms $b:C$, and $y::C$. We need to show that item 1 holds for the term $bIy$. Using induction on $b:C$ and $y::C$, by items 4 and 5, we have  $\Box_{\mathcal{D}} (b)  +  \Box_{\mathcal{D}} (C) \leq \Box_{\mathcal{D}}(\mathcal{A})+1$, and $ -\Box_{\mathcal{D}} (y)  -  \Box_{\mathcal{D}} (C) \leq 0$.  By adding these inequalities, we get  $\Box_{\mathcal{D}} (b) - \Box_{\mathcal{D}} (y) \leq  \Box_{\mathcal{D}}(\mathcal{A})+1$. Similarly, by items 4 and 5, we have  $\Diamond_{\mathcal{D}} (y)  +  \Diamond_{\mathcal{D}} (C) \leq \Diamond_{\mathcal{D}}(\mathcal{A})+1$, and $-\Diamond_{\mathcal{D}} (b)  - \Diamond_{\mathcal{D}} (C) \leq 0$.  By adding these inequalities, we get  $\Diamond_{\mathcal{D}} (y) - \Diamond_{\mathcal{D}} (b) \leq  \Diamond_{\mathcal{D}}(\mathcal{A})+1$.

\textbf{Rules $\wedge_A$,  $\wedge_X$, $\wedge_A^{-1}$,   $\wedge_X^{-1}$:}  immediate from the definitions and the induction hypothesis.

\textbf{$I$-compatibility rules:}  we only give a proof for the rule $\Box y$. The other proofs are analogous. We  need to prove that item 2 holds for the newly added term $b R_\Box y$. By induction applied to the term $b I \Box y$ from item 1, we get 
$\Box_{\mathcal{D}} (b) -  \Box_{\mathcal{D}} (\Box y)= \Box_{\mathcal{D}} (b) -  \Box_{\mathcal{D}} ( y)
+1
\leq \Box_{\mathcal{D}}(\mathcal{A})+1$, and $\Diamond_{\mathcal{D}} (\Box y) -  \Diamond_{\mathcal{D}} (b) =\Diamond_{\mathcal{D}} (y) -  \Diamond_{\mathcal{D}} (b)
\leq \Diamond_{\mathcal{D}}(\mathcal{A})+1$. 

\textbf{$\Box$ and $\Diamond$ rules:}  we only give the proof for the $\Box$ case, as the proof for the $\Diamond$ rule is analogous. Using induction on $b:[R_\Box] C$ and $y::C$, by items 4 and 5, we have  $ \Box_{\mathcal{D}} (b)  +  \Box_{\mathcal{D}} ([R_\Box]C)  = \Box_{\mathcal{D}} (b)  +  \Box_{\mathcal{D}} (C) +1\leq \Box_{\mathcal{D}}(\mathcal{A})+1$,  and $ -\Box_{\mathcal{D}} (y)  -  \Box_{\mathcal{D}} (C) \leq 0$.  By adding them, we get $\Box_{\mathcal{D}} (b) +1 - \Box_{\mathcal{D}} (y) \leq  \Box_{\mathcal{D}}(\mathcal{A})+1$. Similarly, 
by items 4 and 5 we have $-\Diamond_{\mathcal{D}} (b)- \Diamond_{\mathcal{D}} ([R_\Box] C)=-\Diamond_{\mathcal{D}} (b)- \Diamond_{\mathcal{D}} (C) \leq 0$, and  $\Diamond_{\mathcal{D}} (y)+ \Diamond_{\mathcal{D}} (C) \leq \Diamond_{\mathcal{D}}(\mathcal{A})+1$. Hence,
$\Diamond_{\mathcal{D}} (y) - \Diamond_{\mathcal{D}} (b) \leq \Diamond_{\mathcal{D}}(\mathcal{A})+1$.

\textbf{Adjunction rules $R_\Box $ and $R_\Diamond$:} we only give the proof for $R_\Box$. The proof for the  $R_\Diamond$ case is analogous. 
We need to prove item 1 for the terms $\Diamondblack b I y$ and $b I \Box y$. By the induction hypothesis applied to $b R_\Box y$, by item 2 we have $\Box_{\mathcal{D}} (b) +1 -\Box_{\mathcal{D}} (y) = \Box_{\mathcal{D}} (\Diamondblack b)  -\Box_{\mathcal{D}} (y) = \Box_{\mathcal{D}} (b)  -\Box_{\mathcal{D}} (\Box y) \leq \Box_{\mathcal{D}}(\mathcal{A})+1$. Similarly, we also have $\Diamond_{\mathcal{D}} (y)- \Diamond_{\mathcal{D}} (b) = \Diamond_{\mathcal{D}} (y)- \Diamond_{\mathcal{D}} (\Diamondblack b) =\Diamond_{\mathcal{D}} (\Box y)- \Diamond_{\mathcal{D}} (b) \leq   \Diamond_{\mathcal{D}}(\mathcal{A})+1$.

\textbf{Appending rules $a_C$ and $x_C$:} we only give the proof for $x_C$. The proof for $a_C$  is analogous. We add a new term $b:C$ from the term $b I x_C$. As $\Box_{\mathcal{D}}(x_C)=- \Box_{\mathcal{D}}(C) $ and $\Diamond_{\mathcal{D}}(x_C)=- \Diamond_{\mathcal{D}}(C) $,
by the induction hypothesis and using Equation \eqref{eq:IH 3} for $x_C$, we immediately get the required conditions on $\Box_{\mathcal{D}}(C)$ and $\Diamond_{\mathcal{D}}(C)$. Moreover, by induction on term $b I x_C$ using item 1, we have  $-\Box_{\mathcal{D}}(\mathcal{A})-1 \leq \Box_{\mathcal{D}} (b) -  \Box_{\mathcal{D}} (x_C)  =\Box_{\mathcal{D}} (b) +  \Box_{\mathcal{D}} (x_C)  \leq \Box_{\mathcal{D}}(\mathcal{A})+1$ and $-\Diamond_{\mathcal{D}}(\mathcal{A} )-1 \leq  \Diamond_{\mathcal{D}} (x_C) -  \Diamond_{\mathcal{D}} (b) =  -\Diamond_{\mathcal{D}} (C) -  \Diamond_{\mathcal{D}} (b)   \leq \Diamond_{\mathcal{D}}(\mathcal{A})+1$.

\end{proof}

\begin{definition}\label{def:size}
For any concept ABox term of the form $t\equiv a:C$ or $t\equiv x::C$,  $size(t)= 1+|sub(C)|$.  For any relational term $\beta$, $size(\beta)= 2$. For any LE-$\mathcal{ALC}$ ABox $\mathcal{A}$, $size(\mathcal{A})=\sum_{t \in \mathcal{A}} size(t)$. 
\end{definition}
\begin{theorem}[Termination]\label{thm:termination}
For any ABox $\mathcal{A}$, the tableaux algorithm \ref{alg:main algo} terminates in a finite number of steps which is polynomial in $size(\mathcal{A})$.
\end{theorem}
\begin{proof}

New individuals are added to the tableau only in the following ways:

\noindent(1) individuals of the form $a_C$ or $x_C$ can be added by creation rules;

\noindent(2)  individuals of the form $\Box y$, $\blacksquare y$, $\Diamond b$, and $\Diamondblack b$ can be added through the expansion rules for $b R_\Box x$ and $y R_\Diamond a$. 

New concepts can only be added by the appending rules through some constant that was already added. Note that no new propositional connective is added by any of the rules. Thus, the only concept  that can appear are added by the application of $\Box$ and $\Diamond$  operators to concepts already appearing in $\mathcal{A}$. By Equation \eqref{eq:IH 1} of Lemma \ref{lem:depth}, the maximum number of $\Box$ or $\Diamond$ connectives appearing in any concept  added is bounded by $\Box_{\mathcal{D}}(\mathcal{A})+1$ and $\Diamond_{\mathcal{D}}(\mathcal{A})+1$, respectively. Also, by Equations \eqref{eq:IH 2} and  \eqref{eq:IH 3} of Lemma \ref{lem:depth} the number of $\Diamondblack$ (resp.~$\Diamond$, $\Box$, $\blacksquare$),  occurring in a constant $b$ (resp.~$b$, $y$, $y$) is bounded by $\Box_{\mathcal{D}}(\mathcal{A})+1$ (resp.~$\Diamond_{\mathcal{D}}(\mathcal{A})+1$, $\Box_{\mathcal{D}}(\mathcal{A})+1$, $\Diamond_{\mathcal{D}}(\mathcal{A})+1$).  Therefore, 
  the total number of new constant and concepts occurring in an expansion of $\mathcal{A}$ is bounded by $size(\mathcal{A})* (\Box_{\mathcal{D}}(\mathcal{A}) + \Diamond_{\mathcal{D}}(\mathcal{A}))$. Thus, only finitely many individuals of type (2) are added. Overall, the size of the tableau expansion (and hence the model) is $O((size(\mathcal{A})* (\Box_{\mathcal{D}}(\mathcal{A}) + \Diamond_{\mathcal{D}}(\mathcal{A}))^2*(|\mathcal{R}_\Box|+|\mathcal{R}_\Diamond|) )$. 
Since the tableaux algorithm for LE-$\mathcal{ALC}$ does not involve any branching, the above theorem implies that the time complexity of checking the consistency of an LE-$\mathcal{ALC}$ ABox $\mathcal{A}$  using the tableaux algorithm is $Poly(size(\mathcal{A}))$.
\end{proof}

 \subsection{Soundness  of the tableaux algorithm}\label{ssec:Soundness}
 For any consistent ABox $\mathcal{A}$, we let its {\em completion} $\overline{\mathcal{A}}$ be its maximal expansion (which exists due to termination) after post-processing. 
 If there is no clash in  $\overline{\mathcal{A}}$,  we construct a model $(\mathbb{F},\cdot^\mathrm{I})$ where $A$ and $X$ are the sets of names of objects and features occurring in the expansion, and for any $a \in A$, $x \in X$, and any role names $R_\Box \in \mathcal{R}_\Box$, $R_\Diamond \in \mathcal{R}_\Diamond$
we have $a I x$, $a R_\Box x$, $x R_\Diamond a$ iff such relational terms explicitly occur in  $\overline{\mathcal{A}}$. We also add a new element  $x_\bot$ (resp.~$a_\top$) to $X$ (resp.~$A$) such that it is not related to any element of $A$ (resp.~$X$) by any relation. Let $\mathbb{F}= (A,X,I, \mathcal{R}_\Box, \mathcal{R}_\Diamond)$ be the relational structure  obtained in this manner. We define an interpretation $\mathrm{I}$ on it as follows. For any object name $a$, and feature name $x$,  we let $a^\mathrm{I}\coloneqq a$ and $x^\mathrm{I}\coloneqq x$. For any atomic
concept $D$,  we  define  $D^{\mathrm{I}}= ({x_D}^\downarrow, {a_D}^\uparrow)$. Next, we show that $\mathrm{I}$ is  a valid interpretation for LE-$\mathcal{ALC}$. To this end, we need to show that $\mathbb{F}$ is an enriched formal context, i.e.~that all $R_\Box$ and $R_\Diamond$ are $I$-compatible, and that $D^{\mathrm{I}}$ is  a concept in the concept lattice $\mathbb{P}^+$ of $\mathbb{P} = (A, X, I)$. The latter condition is shown in the next lemma, and the former in the subsequent one.
 \begin{lemma}\label{lem:atomic concepts}
$x_D^{\downarrow\uparrow}=a_D^\uparrow$ and $a_D^{\uparrow\downarrow}=x_D^\downarrow$ for any $D \in \mathcal{C}$.
\end{lemma}
\begin{proof}
For any atomic concept $D$, we have   $a_D I x_D$  by the creation and appending rules. Therefore, $I^{(0)}[I^{(1)}[a_D]] \subseteq I^{(0)}[x_D]$.  Suppose  $a_D I y$ and $b I x_D$ are in $\overline{\mathcal{A}}$, then by the appending and basic rule we get $b I y \in \overline{\mathcal{A}}$. Therefore, $I^{(0)}[x_D]\subseteq I^{(0)}[I^{(1)}[a_D]] $. Hence the statement is proved.

\end{proof}

\begin{lemma} \label{lem:Galois-stability}
All the relations    $R_\Box \in \mathcal{R}_\Box$ and $R_\Diamond \in \mathcal{R}_\Diamond$ in $\mathbb{F}= (\mathbb{P}, \mathcal{R}_\Box, \mathcal{R}_\Diamond)$ are $I$-compatible.
\end{lemma}
\begin{proof} 
We prove that $ R_\Box^{(0)}[y]$ is Galois-stable for every $y$ appearing in $\overline{\mathcal{A}}$. All the remaining proofs are similar. Consider the case where $\Box y$ appears in  the tableaux expansion: if $b I \Box y$ (resp.~$b R_\Box y)$ is in $\overline{\mathcal{A}}$,
 then $b R_\Box y$ (resp.~$b I \Box y$) is added by the rule $\Box y$ (resp.~$R_\Box$).

 In the case where $\Box y$  does not appear in $\overline{\mathcal{A}}$, then there is no term of the form $b R_\Box y$ in $\overline{\mathcal{A}}$. Therefore, we have that $R_\Box^{(0)}[y] = \emptyset=I^{(0)}[X]$ is Galois-stable, because we have a feature $x_\bot$ which is not related to any of the objects. 
\end{proof}

From the  lemmas above, it immediately follows that  the tuple $M =(\mathbb{F}, \cdot^\mathrm{I})$, with $\mathbb{F} $ and $\cdot^\mathrm{I}$ defined at the beginning of the present section, is a model for LE-$\mathcal{ALC}$. The following lemma  states that the interpretation of any concept $C$ in the model $M$ is completely determined by the terms of the form $b I x_C$ and $a_C I y$ occurring in the tableau expansion.

\begin{lemma} \label{lem:Soundness pre}
Let  $M =(\mathbb{F}, \cdot^\mathrm{I})$ be the model defined by the construction
above.  Then for any concept $C$  and individuals $b$, $y$ occurring in ${\mathcal{A}}$, 

(1) $b \in \val{C}_M$ iff \ $b I x_C \in  \overline{\mathcal{A}}$ \quad\quad (2) $y \in \descr{C}_M$ iff \ $a_C I  y\in  \overline{\mathcal{A}}$.
\end{lemma}
\begin{proof}
 The proof is by simultaneous induction on the complexity of $C$. The base case is obvious by the construction of $M$. 

For the induction step, we distinguish four cases.

\noindent 1.~Suppose $C=C_1 \wedge C_2$. 

For the first claim, notice that $b \in \val{C_1 \wedge C_2}$ iff $b \in \val{C_1}$, and $b \in \val{C_2}$.  By the induction hypothesis, this is equivalent to $b Ix_{C_1}, b Ix_{C_2} \in 
\overline{\mathcal{A}}$.  Suppose $b \in \val{C_1 \wedge C_2}$. Then $b Ix_{C_1}, b Ix_{C_2} \in 
\overline{\mathcal{A}}$. By the appending rules, $b:C_1, b:C_2 \in \overline{\mathcal{A}} $. As $C_1 \wedge C_2$ appears in $\overline{\mathcal{A}}$, by rule $\wedge_A^{-1}$, $b:C_1 \wedge C_2$ is added which by the creation and appending rule implies $b I x_{C_1 \wedge C_2} \in \overline{\mathcal{A}}$. Conversely, suppose $b I x_{C_1 \wedge C_2} \in \overline{\mathcal{A}}$.  By the appending  rule we have  $b:C_1 \wedge C_2 \in  \overline{\mathcal{A}}$. By rule $\wedge_A$, we have  $b:C_1, b:C_2 \in  \overline{\mathcal{A}}$. By the creation and basic rule we have  $b I x_{C_1}, b I x_{C_2} \in  \overline{\mathcal{A}}$. By induction hypothesis, this implies $b \in \val{C_1}$, and  $b \in \val{C_2}$. By definition, this implies $b \in \val{C_1 \wedge C_2}$.

To prove item 2, $y \in \descr{C_1 \wedge C_2}$ iff  $\forall b (b \in \val{C_1 \wedge C_2} \implies b I y \in \overline{\mathcal{A}})$.  By the proof above, this is equivalent to $\forall b (b I x_{C_1 \wedge C_2} \in \overline{\mathcal{A}} \implies b I y \in \overline{\mathcal{A}} )$. Suppose  $y \in \descr{C_1 \wedge C_2}$. By creation and basic rule, $ a_{C_1 \wedge C_2} I  x_{C_1 \wedge C_2} \in \overline{\mathcal{A}}$. By $\forall b (b I x_{C_1 \wedge C_2} \in \overline{\mathcal{A}} \implies b I y \in \overline{\mathcal{A}} )$, we get $a_{C_1 \wedge C_2} I  y$. Conversely, suppose $a_{C_1 \wedge C_2} I  y \in \overline{\mathcal{A}} $. From $b I x_{C_1 \wedge C_2} \in \overline{\mathcal{A}}$ and by the appending and basic rule, we get $b I y \in \overline{\mathcal{A}}$. Hence the statement is proved. 

\noindent 2.~The proof for $C=C_1 \vee C_2$ is similar to the previous case.

\noindent 3. Suppose $C=[R_\Box]C_1$.

  For the first claim, notice that $b \in \val{[R_\Box]C_1}$ iff $\forall y (y \in \descr{C} \implies b R_\Box y)$.  By the induction hypothesis, this is equivalent to $\forall y (a_C I y\in \overline{\mathcal{A}} \implies b R_\Box y \in \overline{\mathcal{A}})$.  Suppose $b \in \val{[R_\Box]C_1}$. By the creation and basic rule, $a_{C_1} I x_{C_1} \in \overline{\mathcal{A}}$.  Then by $\forall y (a_C I y\in \overline{\mathcal{A}} \implies b R_\Box y \in \overline{\mathcal{A}})$, we have $b R_\Box x_{C_1} \in \overline{\mathcal{A}}$, meaning that via rule $R_\Box$ we get $b I  x_{\Box C_1} \in \overline{\mathcal{A}}$.  Conversely, suppose $b I x_{\Box C_1}$ and $a_{C_1} I y$ occur in $\overline{\mathcal{A}}$.  By rule $\Box y$, we have $\Diamondblack b I x_{C_1} \in \overline{\mathcal{A}}$. By the appending and basic rule, $\Diamondblack b I y \in \overline{\mathcal{A}}$. By rule $\Diamondblack b$, we get $b R_\Box y \in \overline{\mathcal{A}}$.

To prove item 2, $y \in \descr{[R_\Box]C_1}$ iff  $\forall b (b \in \val{[R_\Box]C_1} \implies b I y )$.  By the proof above, this is equivalent to $\forall b (b I x_{[R_\Box]C_1} \in \overline{\mathcal{A}} \implies b I y \in \overline{\mathcal{A}} )$. Suppose  $y \in \descr{[R_\Box]C_1}$. By the creation and basic rule we have $a_{[R_\Box]C_1} I x_{[R_\Box]C_1} \in \overline{\mathcal{A}} $. Therefore, by  $\forall b (b I x_{[R_\Box]C_1} \in \overline{\mathcal{A}} \implies b I y \in \overline{\mathcal{A}} )$, we have $a_{[R_\Box]C_1} I y \in \overline{\mathcal{A}} $. Conversely, suppose $a_{[R_\Box]C_1} I y \in \overline{\mathcal{A}} $ and $b I  x_{[R_\Box]C_1} \in \overline{\mathcal{A}} $. By the appending and basic rule,  we get $b I y \in \overline{\mathcal{A}}$.  Hence the statement is proved.

\noindent 4. The proof for  $C=\langle R_\Diamond \rangle C_1$  is similar to the previous one. 
\end{proof}

\begin{theorem}[Soundness]\label{thm:Soundness}
    The model $M =(\mathbb{F}, \cdot^\mathrm{I})$ defined above satisfies the  ABox $\mathcal{A}$. 
\end{theorem}
\begin{proof}
We proceed by cases.

\noindent 1. By construction, $M$ satisfies all terms of the form $b R_\Box y$, $b I y$, or $y R_\Diamond b$  in $\mathcal{A}$.

\noindent 2.   By construction, any relational term is satisfied by  $M$ iff it explicitly occurs in $\overline{\mathcal{A}}$. Thus, either $M$ satisfies  all  terms of the form $\neg (b R_\Box y)$, $\neg (b I y)$, and $\neg (y R_\Diamond b)$ occurring in $\mathcal{A}$, 
 or some expansion of  ${\mathcal{A}}$  contains a clash. 

\noindent 3. For the terms of the form $b:C$, $y::C$, $\neg (b:C)$, or $\neg (y::C)$, we have $b \in \val{C}$ iff $b I x_C \in \overline{\mathcal{A}}$, and $y \in \descr C$ iff $a_C I y \in \overline{\mathcal{A}}$ (Lemma \ref{lem:Soundness pre}).  For any $b:C$, $y::C$, $\neg (b:C)$, or $\neg (y::C)$ occurring in $\mathcal{A}$, we respectively add  $b I x_C$, $a_C I y$, $\neg (b I x_C)$, or $\neg (a_C I y) $ to $\overline{\mathcal{A}}$ via the expansion rules, and thus $M$ satisfies the constraints.
\end{proof}

The following corollary is an immediate consequence of the termination and soundness of the tableau procedure.
\begin{corollary}[Finite Model Property]
For any  consistent LE-$\mathcal{ALC}$ ABox $\mathcal{A}$, some model of $\mathcal{A}$ exists  the size of which is polynomial in $size(\mathcal{A})$.
\end{corollary}
\begin{proof}
The model $M$ of Theorem \ref{thm:Soundness} is the required witness. The polynomial bound on the size of $M$ follows from the proof of Theorem \ref{thm:termination}.
\end{proof}

\subsection{Completeness of the tableaux algorithm}\label{ssec:Completeness}
In this section, we  prove the completeness of the tableau algorithm. The following lemma is key to this end, since it shows that  every model for an LE-$\mathcal{ALC}$ ABox  can be extended to a model with classifying object and features.

\begin{lemma}\label{lem:characteristic consistency}

For any ABox $\mathcal{A}$,  any model $M=(\mathbb{F}, \cdot^{\mathrm{I}})$ of $\mathcal{A}$ can be extended to a model $M'=(\mathbb{F}', \cdot^{\mathrm{I}'})$ such that $\mathbb{F}'=(A',X',I',\{R_\Box'\}_{\Box\in\mathcal{G}}, \{R_\Diamond'\}_{\Diamond\in\mathcal{F}})$, $A \subseteq A'$ and $X \subseteq  X'$, and moreover for every $\Box \in \mathcal{G}$ and $\Diamond \in \mathcal{F}$:

1. There exists $a_C \in A'$ and $x_C \in X'$  such that:
\begin{equation}\label{eq:completness 1}
   C^{\mathrm{I}'} =(I'^{(0)}[x_C^{\mathrm{I}'}], I'^{(1)}[a_C^{\mathrm{I}'}]), \quad a_C^{\mathrm{I}'} \in \val{C^{\mathrm{I}'}}, \quad  x_C^{\mathrm{I}'} \in \descr{C^{\mathrm{I}'}},  
\end{equation}

 2. For every individual $b$ in $A$ there exist $\Diamond b$ and $\Diamondblack b$ in $A'$ such that:
\begin{equation}\label{eq:completness 2}
I'^{(1)}[\Diamondblack b] = R_\Box'^{(1)}[b^{\mathrm{I}'}] \quad \mbox{and} \quad I'^{(1)}[\Diamond b] = R_\Diamond'^{(0)}[b^{\mathrm{I}'}],
\end{equation}

3. For every individual $y$ in $X$ there exist $\Box y$ and $\blacksquare y$ in $X'$ such that:
\begin{equation}\label{eq:completness 3}
I'^{(0)}[\blacksquare y] = R_\Diamond'^{(1)}[y^{\mathrm{I}'}] \quad \mbox{and} \quad I'^{(0)}[\Box y] = R_\Box'^{(0)}[y^{\mathrm{I}'}]. 
\end{equation}

4. For any $C$, $\val{C^\mathrm{I}}= \val{C^\mathrm{I'}} \cap A$ and $\descr{C^\mathrm{I}}= \descr{C^\mathrm{I'}} \cap X$. 
\end{lemma}
\begin{proof}
Fix $\Box \in \mathcal{G}$ and $\Diamond \in \mathcal{F}$. Let $M'$ be defined as follows. For every concept $C$,  we add new elements $a_C$ and $x_C$ to $A$ and $X$ (respectively) to obtain the sets $A'$ and $X'$. 
For any  $J \in \{I, R_\Box\}$,   any $a \in A'$ and $x \in X'$, we set $a J' x$ iff one of the following holds:

    \noindent 1. $a \in A$, $x \in X$, and $a J x$; 
    
    \noindent 2. $x \in X$, and $a=a_C$ for some concept $C$, and $b J x$ for all $b \in \val{C^\mathrm{I}}$;
    
    \noindent 3.   $a \in A$, and $x=x_C$ for some concept $C$, and $a J y$ for all $y \in \descr{C^\mathrm{I}}$;
    
    \noindent 4.  $a=a_{C_1}$ and $x=x_{C_2}$ for some $C_1$, $C_2$, and $b J y$ for all $b \in  \val{C_1^\mathrm{I}}$, and $ y \in  \descr{C_2^\mathrm{I}}$.
    
We set $x R_\Diamond' a$ iff 
one of the following holds:

    \noindent 1.  $a \in A$, $x \in X$, and $x R_\Diamond a$; 
    
    \noindent 2. $x \in X$, and $a=a_C$ for some concept $C$, and $x R_\Diamond b $ for all $b \in \val{C^\mathrm{I}}$;
    
    \noindent 3. $a \in A$, and $x=x_C$ for some concept $C$, and $y R_\Diamond a$ for all $y \in \descr{C^\mathrm{I}}$;
     
    \noindent 4.  $a=a_{C_1}$ and $x=x_{C_2}$ for some  $C_1$, $C_2$, and $y R_\Diamond b$ for all $b \in  \val{C_1^\mathrm{I}}$,  $ y \in  \descr{C_2^\mathrm{I}}$.

\noindent For any $b \in A$, $y \in X$, let  $\Diamondblack b =a_{\Diamondblack(cl(b))}$, $\Diamond b =a_{\Diamond(cl(b))}$,  $\blacksquare y=x_{\blacksquare (cl(y))}$, and   $\Box  y=x_{\Box (cl(y))}$, where $cl(b)$ (resp.~$cl(y)$) is the smallest concept generated by $b$ (resp.~$y$), and  the operations  $\Diamondblack$ and $\blacksquare$ are the adjoints of operations $\Box$ and $\Diamond$, respectively. For any  $C$, let $C^{\mathrm{I'}} = (I'^{(0)}[x_C], I'^{(1)}[a_C]) $. Then $M'$ is as required.
\end{proof}
\begin{theorem}[Completeness]\label{thm:completeness}
    Let $\mathcal{A}$ be a consistent ABox and $\mathcal{A}'$ be obtained via the application of any expansion rule or post-processing  applied to $\mathcal{A}$. Then $\mathcal{A}'$ is also consistent. 
\end{theorem}
\begin{proof}
If $\mathcal{A}$ is consistent, by  Lemma \ref{lem:characteristic consistency},  a model $M'$ of $\mathcal{A}$ exists which satisfies \eqref{eq:completness 1}, \eqref{eq:completness 2} and \eqref{eq:completness 3}. The statement follows from the fact that 
any term added by any expansion rule or in post-processing is satisfied by $M'$ where we interpret $a_C$, $x_C$, $\Diamondblack b$, $\Diamond b$, $\Box y$, $\blacksquare y$ as in Lemma \ref{lem:characteristic consistency}. 
\end{proof}

\begin{remark}
    The algorithm can easily be extended to acyclic TBoxes, via the unravelling technique (cf.~\cite{DLbook} for details). Notice that in the presence of TBoxes that are not completely unravelled (cf.~Subsection \ref{ssec:Description logic ALC}), polynomial-time complexity for the consistency check procedure is not necessarily preserved. 
\end{remark}

\section{Conclusion and future work}
\label{Sec: conclusions}

In this paper,  we define a two-sorted non-distributive description logic LE-$\mathcal{ALC}$ to describe and reason about formal concepts arising from  (enriched) formal contexts from FCA. 
We describe ABox and TBox terms for the logic and define a tableaux algorithm for it. This tableaux algorithm decides the consistency of ABoxes and acyclic TBoxes, and provides a procedure to construct a model when the input is consistent. We show that this algorithm is computationally more efficient than the tableaux algorithm for $\mathcal{ALC}$.

This work can be extended in several interesting directions.

\fakeparagraph{Dealing with cyclic TBoxes and  RBox axioms.} 
In this paper, we introduced a tableaux algorithm only for knowledge bases with acyclic TBoxes. We conjecture that the following statement holds of general (i.e.~possibly cyclic) TBoxes.

Developing such an algorithm is a research direction we are currently pursuing. Another aspect we intend to develop   in future work concerns giving a complete axiomatization for LE-$\mathcal{ALC}$. RBox axioms are used in  description logics to describe the relationship between different relations in  knowledge bases and  the properties of these relations such as reflexivity, symmetry, and transitivity.  It would be interesting to see if it is possible to obtain necessary and/or sufficient conditions on the shape of RBox axioms for which a tableaux algorithm can be obtained. This has an interesting relationship with the problem in LE-logic of providing computationally efficient proof systems for various extensions of LE-logic in a modular manner \cite{greco2016unified,ICLArough}. 

\fakeparagraph{Generalizing to  other semantic frameworks.} The  non-distributive DL introduced in this paper is  semantically motivated by a relational  semantics for LE-logics which establishes a link with FCA. 
A different semantics for the same logic, referred to as  graph-based semantics \cite{conradie2020non}, provides another interpretation of the same logic as  a logic suitable for evidential and hyper-constructivist reasoning. In the future, we intend to develop description logics for reasoning in the framework of graph-based semantics, to appropriately model evidential and hyper-constructivist settings. 

\fakeparagraph{Generalizing to more expressive description logics.} The  DL LE-$\mathcal{ALC}$ is the non-distributive counterpart  of  $\mathcal{ALC}$. A natural direction for further research is to explore the non-distributive counterparts of extensions of $\mathcal{ALC}$ such as  $\mathcal{ALCI}$ and $\mathcal{ALCIN}$.

 \fakeparagraph{Description logic and Formal Concept Analysis.} 
 The relationship between FCA and DL has been studied and used in several applications \cite{DLandFCA1,DLandFCA2,DLandFCA3}. The framework of LE-$\mathcal{ALC}$ formally brings  FCA and DL together, both because its concepts   are naturally interpreted as formal concepts in FCA, and because its language   is designed  to represent knowledge and reasoning in   enriched formal contexts. 
 Thus, these results pave the way to the possibility of establishing a closer and more formally explicit  connection between FCA and DL, and of using this connection in  theory and applications.

\bibliographystyle{splncs04}
\bibliography{ref}

\begin{thebibliography}{10}
\providecommand{\url}[1]{\texttt{#1}}
\providecommand{\urlprefix}{URL }
\providecommand{\doi}[1]{https://doi.org/#1}

\bibitem{DLandFCA1}
Atif, J., Hudelot, C., Bloch, I.: Explanatory reasoning for image understanding
  using formal concept analysis and description logics. IEEE Transactions on
  Systems, Man, and Cybernetics: Systems  \textbf{44}(5),  552--570 (2014)

\bibitem{DLhandbook}
Baader, F., Calvanese, D., McGuinness, D., Nardi, D., Patel-Schneider, P.: The
  Description Logic Handbook: Theory, Implementation and Applications.
  Cambridge University Press (2003)

\bibitem{DLbook}
Baader, F., Horrocks, I., Lutz, C., Sattler, U.: An Introduction to Description
  Logic. Cambridge University Press (2017)

\bibitem{DLandFCA2}
Baader, F., Sertkaya, B.: Applying formal concept analysis to description
  logics. Concept Lattices pp. 261--286 (2004)

\bibitem{ICLArough}
van~der Berg, I., De~Domenico, A., Greco, G., Manoorkar, K.B., Palmigiano, A.,
  Panettiere, M.: Labelled calculi for the logics of rough concepts. Logic and
  Its Applications pp. 172--188 (2023)

\bibitem{fuzzyDL1}
Borgwardt, S., Pe{\~{n}}aloza, R.: Fuzzy description logics -- a survey.
  Scalable Uncertainty Management pp. 31--45 (2017)

\bibitem{conradie2022modal}
Conradie, W., De~Domenico, A., Manoorkar, K., Palmigiano, A., Panettiere, M.,
  Prieto, D.P., Tzimoulis, A.: Modal reduction principles across relational
  semantics. arXiv preprint arXiv:2202.00899  (2022)

\bibitem{conradie2021rough}
Conradie, W., Frittella, S., Manoorkar, K., Nazari, S., Palmigiano, A.,
  Tzimoulis, A., Wijnberg, N.M.: Rough concepts. Information Sciences
  \textbf{561},  371--413 (2021)

\bibitem{conradie2017toward}
Conradie, W., Frittella, S., Palmigiano, A., Piazzai, M., Tzimoulis, A.,
  Wijnberg, N.M.: Toward an epistemic-logical theory of categorization.
  Electronic Proceedings in Theoretical Computer Science, EPTCS  \textbf{251}
  (2017)

\bibitem{conradie2016categories}
Conradie, W., Frittella, S., Palmigiano, A., Piazzai, M., Tzimoulis, A.,
  Wijnberg, N.M.: Categories: how {I} learned to stop worrying and love two
  sorts. In: International Workshop on Logic, Language, Information, and
  Computation. pp. 145--164. Springer (2016)

\bibitem{conradie2019algorithmic}
Conradie, W., Palmigiano, A.: Algorithmic correspondence and canonicity for
  non-distributive logics. Annals of Pure and Applied Logic  \textbf{170}(9),
  923--974 (2019)

\bibitem{conradie2020non}
Conradie, W., Palmigiano, A., Robinson, C., Wijnberg, N.: Non-distributive
  logics: from semantics to meaning. In: Rezus, A. (ed.) Contemporary Logic and
  Computing, Landscapes in Logic, vol.~1, pp. 38--86. College Publications
  (2020)

\bibitem{ganter2012formal}
Ganter, B., Wille, R.: Formal concept analysis: mathematical foundations.
  Springer Science \& Business Media (2012)

\bibitem{giordano2015semantic}
Giordano, L., Gliozzi, V., Olivetti, N., Pozzato, G.L.: Semantic
  characterization of rational closure: From propositional logic to description
  logics. Artificial Intelligence  \textbf{226},  1--33 (2015)

\bibitem{giordano2022conditional}
Giordano, L., Gliozzi, V., Theseider~Dupr{\'e}, D.: A conditional, a fuzzy and
  a probabilistic interpretation of self-organizing maps. Journal of Logic and
  Computation  \textbf{32}(2),  178--205 (2022)

\bibitem{greco2016unified}
Greco, G., Ma, M., Palmigiano, A., Tzimoulis, A., Zhao, Z.: Unified
  correspondence as a proof-theoretic tool. Journal of Logic and Computation
  \textbf{28}(7),  1367–1442 (2016)

\bibitem{DLandFCA3}
Jiang, Y.: Semantifying formal concept analysis using description logics.
  Knowledge-Based Systems  \textbf{186} (2019)

\bibitem{non-monotone}
Lieto, A., Pozzato, G.L.: A description logic framework for commonsense
  conceptual combination integrating typicality, probabilities and cognitive
  heuristics. Journal of Experimental \& Theoretical Artificial Intelligence
  \textbf{32}(5),  769--804 (2020)

\bibitem{fuzzyDL2}
Ma, Z.M., Zhang, F., Wang, H., Yan, L.: An overview of fuzzy description logics
  for the semantic web. The Knowledge Engineering Review  \textbf{28}(1),
  1–34 (2013)

\bibitem{depaiva2011constructive}
de~Paiva, V., Haeusler, E.H., Rademaker, A.: Constructive description logics
  hybrid-style. Electronic Notes in Theoretical Computer Science  \textbf{273},
   21--31 (2011)

\bibitem{shilov2007proposal}
Shilov, N.V., Han, S.Y.: A proposal of description logic on concept lattices.
  In: Proceedings of the Fifth International Conference on Concept Lattices and
  their Applications. pp. 165--176 (2007)

\bibitem{wurm2017language}
Wurm, C.: Language-theoretic and finite relation models for the (full) lambek
  calculus. Journal of Logic, Language and Information  \textbf{26}(2),
  179--214 (2017)

\end{thebibliography}

\end{document}